\documentclass[12pt]{article}
\usepackage{amsfonts}
\usepackage{mathrsfs}
\usepackage{amssymb,amsmath,mathrsfs, amsthm}
\usepackage{amssymb, palatino}
\usepackage{graphicx}
\usepackage{mathtools,cite}
\usepackage{hyperref}

\numberwithin{equation}{section}

\newtheorem{theorem}{Theorem}[section]
\newtheorem{lemma}{Lemma}[section]
\newtheorem{corollary}{Corollary}[section]
\newtheorem{proposition}{Proposition}[section]
\theoremstyle{definition}
\newtheorem{definition}{Definition}[section]
\newtheorem{remark}{Remark}[section]
\newtheorem{example}{Example}[section]
\theoremstyle{remark}

\date{}
\begin{document}
\title{Schwarz lemma on polydiscs endowed with holomorphic invariant K\"ahler-Berwald metrics}
\author{Shuqing Lin (sqlinmath@126.com)\\
School of Mathematical Sciences, Xiamen
University\\ Xiamen 361005, China\\
Liling Sun (sunliling@yeah.net)\\
School of Sciences, Jimei University\\ Xiamen 361021, China\\
Chunping Zhong (zcp@xmu.edu.cn)\\
School of Mathematical Sciences, Xiamen
University\\ Xiamen 361005, China
}

\date{}

\maketitle
\begin{abstract}
 In this paper,
 we obtain a Schwarz lemma for holomorphic mappings from the unit polydisc $P_m$ into the unit polydisc $P_n$, here $P_m$ and $P_n$ are endowed with $\mbox{Aut}(P_m)$-invariant K\"ahelr-Berwald metric $F_{t,k}$ and $\mbox{Aut}(P_n)$-invariant K\"ahler-Berwald metric $\tilde{F}_{\tilde{t},\tilde{k}}$ respectively. Our result generalizes the Schwarz lemma for holomorphic mappings from $P_m$ into $P_n$ whenever $P_m$ and $P_n$ are endowed with the Bergman metrics respectively. We also obtain a distortion theorem on the unit polydisc $P_m$, where $P_m$ is endowd with an $\mbox{Aut}(P_m)$-invariant K\"ahler-Berwald metric $F_{t,k}$, and show that for each fixed $t\in[0,+\infty)$ and integer $k\geq 2$, $F_{t,k}$ is actually a K\"ahler Finsler-Einstein metric in the sense of T. Aikou.
\end{abstract}
\textbf{Keywords:} Schwarz lemma; K\"ahler-Berwald metric; Distortion theorem; Finsler-Einstein metric.\\
\textbf{MSC(2010):}32H02, 53C60, 53C56. \\
\section{Introduction and statement of main results}\label{se1}

Let $\triangle=\{z\in\mathbb{C}: |z|<1\}$ be the  unit disk in $\mathbb{C}$ endowed with the Poincar\'e metric
\begin{equation}
ds^2=\frac{dzd\overline{z}}{(1-|z|^2)^2},\quad z\in\triangle,\label{pm}
\end{equation}
and $\mbox{Aut}(\triangle)$ denotes the group of holomorphic automorphisms of $\triangle$.
The classical Schwarz lemma states that every holomorphic mapping $f$ from $\triangle$ into itself decreases the Poincar\'e metric, i.e. $
f^\ast ds^2\leq ds^2,$
where $f^\ast ds^2$ denotes the pull-back metric. An immediate consequence of the classical Schwarz lemma is that each $f\in\mbox{Aut}(\triangle)$ is an isometry of $ds^2$.
Let $f:\triangle\rightarrow N$ be any holomorphic mapping, where $N$ is a Riemann surface endowed with a conformal metric $d\sigma^2=\rho(w)|dw|^2$ such that its Gaussian curvature satisfies
 $K=-\frac{2}{\rho}\frac{\partial^2\log\rho}{\partial w\partial\overline{w}}\leq-4.$ In \cite{Ah}, L. V. Ahlfors proved that $f^\ast d\sigma^2\leq ds^2.$
Ahlfors' generalization of the classical Schwarz lemma uncovered the relationship between curvatures of the differential metrics and the decreasing property of the pull-back metric $f^\ast d\sigma^2$ under holomorphic mapping $f:\triangle\rightarrow N$. This opened the door of generalizing Schwarz lemma of interior type from the viewpoint of differential geometry. For the generalizations along this line we  refer to
K. H. Look \cite{Look1,Look2}, S. Kobayashi \cite{K-1}, S. S. Chern \cite{Chern}, S. T. Yau \cite{Yau}, Z. H. Chen, S. Y. Cheng and Q. K. Lu\cite{CCL}, H. L. Royden \cite{Royden}, H. C. Yang and Z. H. Chen \cite{YC}, V. Tosatti \cite{Tosatti}, K. T. Kim and H. J. Lee \cite{KL} and references therein. There are also lots of interest in investigating Schwarz lemma of boundary type, we refer to D. M. Burns and S. G. Krantz \cite{BK}, T. S. Liu, J. F. Wang and X. M. Tang \cite{LWT}, T. S. Liu and X. M. Tang \cite{LT}, Y. Liu, Z. H. Chen and Y. F. Pan \cite{LCP}, X. M. Tang, T. S. Liu and W. J. Zhang \cite{TLZ}, J. F. Wang, T. S. Liu and X. M. Tang\cite{WLT}.

Recently, the generalization of Schwarz lemma  to strongly pseudoconvex complex Finsler manifolds attracts some interest. In \cite{SS}, B. Shen and Y. B. Shen obtained a Schwarz lemma for holomorphic mappings between two compact strongly pseudoconvex complex Finsler manifolds.
 In \cite{Wan}, X. Y. Wan obtained a Schwarz lemma for holomorphic mappings from a complete Riemann surface endowed with a conformal metric into a complex manifold endowed with a strongly pseudoconvex complex Finsler metric. In \cite{NZ-1}, J. Nie and C. P. Zhong obtained a Schwarz lemma for holomorphic mappings from a complete K\"ahler manifold into a strongly pseudoconvex complex Finsler manifold with some curvatures assumptions, and then in \cite{NZ-2}, J. Nie and C. P. Zhong obtained a Schwarz lemma on a complex manifold which admits a  weakly K\"ahler-Finsler metric with some assumptions of the radial flag curvatures and holomorphic sectional curvatures.

In \cite{Zh1}, the third author of this paper proved that there exists no non-Hermitian quadratic $U(m)$-invariant K\"ahler-Finsler metric on any $U(m)$-invariant domain $D\subset \mathbb{C}^m$. An immediate consequence of this result implies that on the  unit ball $B_m$ in $\mathbb{C}^m$, every $\mbox{Aut}(P_m)$-invariant strongly pseudoconvex complex Finsler metric is a constant multiple of the Bergman metric on $B_m$. Very recently in \cite{Zh2}, C. P. Zhong proved that on the  unit polydisc $P_m$ in $\mathbb{C}^m (m\geq 2)$, there exist infinite many $\mbox{Aut}(P_m)$-invariant complex Finsler metrics
\begin{equation}
F_{t, k}(z; v)
=\frac{1}{\sqrt{1+t}} \sqrt{\sum_{l=1}^{m} \frac{|v^l|^2}{(1-|z^l|^2)^2}+t \sqrt[k]{\sum_{l=1}^{m} \frac{|v^l|^{2k}}{(1-|z^l|^2)^{2k}}}} \label{Ftk}
\end{equation}
for any $t\in [0, +\infty)$ and any integer $k\geq 2$, where $z=(z^1,\cdots,z^m)\in P_m\subset\mathbb{C}^m$ and $v=(v^1,\cdots,v^m)\in\mathbb{C}^m\cong T_z^{1,0}P_{m}.$
Moreover, for any fixed $t \in[0,+\infty)$ and any fixed integer $k\geq 2$, $F_{t,k}$ is a complete strongly convex K\"{a}hler-Berwald metric.

Note that when $t=0$, \eqref{Ftk} reduces to the usual Bergman metric
\begin{equation}
	F_0^2(z;v)=\frac{|v^1|^{2}}{\big(1-\left|z^{1}\right|^{2}\big)^{2}}+\cdots
	+\frac{\left|v^m\right|^{2}}{\big(1-\left|z^{m}\right|^{2}\big)^{2}},\quad \forall z\in P_m,v\in T_z^{1,0}P_m\label{F0}
\end{equation}
 on $P_m$, which is clear Hermitian quadratic with respect to the fiber coordinates $v=(v^1,\cdots,v^m)$. For any fixed $t\in (0,+\infty)$ and any fixed integer $k\geq 2$, however, \eqref{Ftk} is a non-Hermitian quadratic metric.

  Let $f:P_m\rightarrow P_m$ be a holomorphic mapping, we denote
 $$f^\ast F_0^2(z;v):=F_0^2(f(z);f_\ast(v)),\quad \forall z\in P_m, v\in T_z^{1,0}P_m,$$
 where $f_\ast$ denotes the differential of $f$ at the point $z$.
 The following theorem is due to K. H. Look \cite{Look1,Look2}.

 \begin{theorem}\label{th1} (\cite{Look1,Look2})~Let $P_{m}$ be the  unit polydisc in $\mathbb{C}^{m}(m\geq 2)$ endowed with the Bergman metric $F_0^2$.~Then for any holomorphic mapping $w=f(z)$ from $P_{m}$ into itself,
 	\begin{equation}
 		\big(f^\ast F_0^2\big)(z; v) \leq m F_0^2(z; v),\quad\forall z\in P_m,v\in T_z^{1,0}P_m.\label{sbm}
 	\end{equation}
 \end{theorem}
 \begin{remark}
 	It is easy to check that any $\mbox{Aut}(P_m)$-invariant Hermitian metric is a positive constant multiple of the Bergman metric on $P_m$. Furthermore, the constant $m$ in  \eqref{sbm} is optimal.
 \end{remark}

Now let $\tilde{z}=(\tilde{z}^1,\cdots,\tilde{z}^n)$  be the complex coordinates on $\mathbb{C}^n(n\geq 2)$ and  $P_n$ be the  unit polydisc in $\mathbb{C}^n$ endowed with the following $\mbox{Aut}(P_n)$-invariant K\"ahler-Berwald metric
\begin{equation}
\tilde{F}_{\tilde{t}, \tilde{k}}(\tilde{z}; \tilde{v})
=\frac{1}{\sqrt{1+\tilde{t}}} \sqrt{\sum_{l=1}^{n} \frac{|\tilde{v}^l|^2}{(1-|\tilde{z}^l|^2)^2}+\tilde{t} \sqrt[\tilde{k}]{\sum_{l=1}^{n} \frac{|\tilde{v}^l|^{2\tilde{k}}}{(1-|\tilde{z}^l|^2)^{2\tilde{k}}}}}, \quad \forall~\tilde{z}\in P_n,\tilde{v} \in T_{\tilde{z}}^{1,0}P_{n}.\label{tFtk}
\end{equation}

The purpose of this paper is first to establish a Schwarz lemma for holomorphic mapping $f:P_{m}\rightarrow P_{n}$, where $P_m$ is endowed with the $\mbox{Aut}(P_m)$-invariant K\"ahler-Berwald metric $F_{t, k}$ given by \eqref{Ftk} and $P_n$ is endowed with the $\mbox{Aut}(P_n)$-invariant K\"ahler-Berwald metric $\tilde{F}_{\tilde{t},\tilde{k}}$ given by \eqref{tFtk}, we then obtain a distortion theorem on $P_m$ whenever it is endowed with the $\mbox{Aut}(P_m)$-invariat metric $F_{t,k}$, and finally show that for any fixed $t\in[0,+\infty)$ and any fixed integer $k\geq 2$, $(P_m,F_{t,k})$ is a K\"ahler Finsler-Einstein manifold in the sense of T. Aikou \cite{Ai2}.

Notice that if $m=n=1$, then $F_{t,k}$ and $\tilde{F}_{\tilde{t},\tilde{k}}$ reduce to the well-known Poincar$\acute{\mbox{e}}$ metric on the  unit disc $\triangle$ in $\mathbb{C}$. In the following, we always assume that $m,n\geq 2$. For any $t, \tilde{t}\in(0,+\infty)$ with $t\neq \tilde{t}$ and $k\neq \tilde{k}$, there exists no holomorphic isometries between the complex Finsler manifolds $(P_m,F_{t,k})$ and $(P_n,\tilde{F}_{\tilde{t},\tilde{k}})$ even if $m=n\geq 2$.

The main results in this paper are  as follows.

\begin{theorem}\label{mth}
 Let $P_{m}$ be the  unit polydisc in $\mathbb{C}^m$ endowed with the $\mbox{Aut}(P_m)$-invariant K\"ahler-Berwald metric $F_{t, k}$, and $P_n$ be the  unit polydisc in $\mathbb{C}^n$ endowed with the $\mbox{Aut}(P_n)$-invariant K\"ahler-Berwald metric $\tilde{F}_{\tilde{t},\tilde{k}}$. Then for any holomorphic mapping $f:P_m\rightarrow P_n$, we have
\begin{equation}
\big(f^\ast \tilde{F}_{\tilde{t},\tilde{k}}^2\big)(z;v)\leq \frac{n+\tilde{t}\sqrt[\tilde{k}]{n}}{1+\tilde{t}}F_{t,k}^2(z;v),\quad \forall z\in P_m, v\in T_z^{1,0}P_m.\label{tii-1}
\end{equation}
The constant
$\frac{n+\tilde{t}\sqrt[\tilde{k}]{n}}{1+\tilde{t}}$
in \eqref{tii-1} is optimal for any fixed $t,\tilde{t}\in[0,+\infty)$ and any fixed integers $k,\tilde{k}\geq 2$.
\end{theorem}

\begin{remark}
(1) For $t=0$ and  $\tilde{t}=0$, $F_{t,k}$ and $\tilde{F}_{\tilde{t},\tilde{k}}$ reduce to the usual Bergman metrics on  $P_m$ and $P_n$ respectively, which are complete K\"ahler metrics, hence Hermitian quadratic complex Finsler metrics. In this case, Theorem \ref{mth} reduces to Theorem \ref{th1} whenever $m=n$.

(2) For any fixed $t,\tilde{t}\in(0,+\infty)$ and any fixed integers $k,\tilde{k}\geq 2$,
$F_{t,k}$ and $\tilde{F}_{\tilde{t},\tilde{k}}$ are holomorphic invariant K\"ahler-Berwald metrics on $P_m$ and $P_n$ respectively, which are non-Hermitian quadratic complex Finsler metrics. Therefore, Theorem \ref{mth} is a natural generalization of Theorem \ref{th1} from Hermitian quadratic metrics to  complex Finsler metrics which are not necessary Hermitian quadratic.

(3) The constant
 $\frac{n+\tilde{t}\sqrt[\tilde{k}]{n}}{1+\tilde{t}}$
 in \eqref{tii-1} is optimal, which depends neither on the parameters $t$ and $k$ in $F_{t,k}$ nor the dimension $m$ of the source polydisc $P_m$, it depends only on the given parameters $\tilde{t}$ and $\tilde{k}$ in $\tilde{F}_{\tilde{t},\tilde{k}}$ and the dimension $n$ of the target polydisc $P_n.$ Moreover, we can view the constant $\frac{n+\tilde{t}\sqrt[\tilde{k}]{n}}{1+\tilde{t}}=\frac{L}{U}$ as the quotient of the lower bound $L=-4$ of the holomorphic sectional curvature of $F_{t,k}$ and the upper bound $U =\frac{-4(1+\tilde{t})}{n+\tilde{t}\sqrt[\tilde{k}]{n}}$ of the holomorphic sectional curvature of  $\tilde{F}_{\tilde{t},\tilde{k}}$ (see Theorem 1.6 in \cite{Zh2}).
\end{remark}

Notice that for every normalized biholomorphic convex mapping $f(z): P_m \rightarrow \mathbb{C}^{m},$ $f_l(z)=f_{l}\left(z^{l}\right)$ is a
  normalized biholomorphic convex function of one complex variable $z^{l}$ on the unit disc $\Delta_l=\Delta,l=1,\cdots,m.$ And we have
\begin{theorem}
	Let $P_m$ be the unit polydisk in $\mathbb{C}^m$ endowed with the $\mbox{Aut}(P_m)$-invariant K\"ahler-Berwald metric $F_{t,k}$ given by \eqref{Ftk}, and $f(z)=(f_1(z),\cdots,f_n(z)): P_m \rightarrow \mathbb{C}^{m}$ be a normalized biholomorphic convex mapping.	
	Then for every $z \in P_m$ and vector $v \in T_z^{1,0}P_m\cong\mathbb{C}^{m}$, the following inequalities
	\begin{equation}
		\left[\frac{1-p(z)}{1+p(z)}\right]^2F_{t,k}^{2}(z;v)
		\leq F_{t,k}^2\left(0; f_{\ast}(z) v\right)
		\leq \left[\frac{1+p(z)}{1-p(z)}\right]^2F_{t,k}^{2}(z; v)\label{1.7}
	\end{equation}
	hold, where $p(z)$ is the Minkowski functional of $P_m$.
	Moreover, one of the equalities holds at some point $z\neq0$ in \eqref{1.7} if and only if $f_l(z^l)=\frac{z^l}{1-e^{i\theta_l}z^l}$ for some $\theta_l\in \mathbb{R}, l=1,\cdots,m.$
\end{theorem}

The following theorem shows that the $\mbox{Aut}(P_m)$-invariant K\"ahler-Berwald metrics $F_{t,k}$ are K\"ahler Finsler-Einstein metrics on $P_m$ in the sense of T. Aikou \cite{Ai2}.
\begin{theorem}
	Let $P_{m}$ be the  unit polydisc in $\mathbb{C}^{m}$ endowed with the $\mbox{Aut}(P_m)$-invariant K\"ahler-Berwald metric $F_{t, k}$ given by \eqref{Ftk}.
	Then for any fixed $t\in[0,+\infty)$ and any fixed integer $k\geq 2$, $(P_m,F_{t,k})$ is a K\"ahler Finsler-Einstein manifold with factor $\varphi=-2$.
\end{theorem}

The contents of the paper are as follows. In section 2, we give some basic notions on complex Finsler metric on domains in $\mathbb{C}^m$. In section 3, we give two lemmas which will be used in section 4 to generalize Schwarz lemma for holomorphic mappings between polydiscs endowed holomorphic invariant K\"ahler-Berwald metrics and derive some corollaries. In section 5, we obtain a distortion theorem on polydiscs endowed with holomorphic invariant K\"ahler-Berwald metrics $F_{t,k}.$ In section 6, we show that $F_{t,k}$ are K\"ahelr Finsler Einstein metrics on polydiscs.

\section{Complex Finsler metric}
Let $D$ be a domain in $\mathbb{C}^m$ and $T^{1,0}D$ be its holomorphic tangent bundle. Denote $z=(z^1,\cdots,z^m)$ the complex coordinates on $D$. A complex tangent vector $\pmb{v}$ in the fiber  $T_z^{1,0}D$ over $z\in D$ can be expressed as
$$\pmb{v}=v^i\frac{\partial}{\partial z^i}\Big|_{z},\quad v=(v^1,\cdots,v^m)\in\mathbb{C}^m,$$
 so that
$(z;v)=(z^1,\cdots,z^m;v^1,\cdots,v^m)$ can be used as global complex coordinates on $T^{1,0}D$. Thus $T^{1,0}D$ is isomorphic to the product bundle $D\times \mathbb{C}^m$, hence a trivial holomorphic vector bundle over $D$.

\begin{definition}\cite{AP}\label{d-cfm}
A (smooth) complex Finsler metric on $D$ is a continuous function $F:T^{1,0}D\rightarrow [0,+\infty)$ satisfying the following properties:

(1) Regularity: $G=F^2$ is $C^\infty$ on $T^{1,0}D\setminus\{\mbox{zero section}\}$;

(2) Positivity: $F(z;v)\geq 0$ for any $(z;v)\in T^{1,0}D$ and $F(z;v)=0$ if and only if $v=0$;

(3) Homogeneity: $F(z;\lambda v)=|\lambda|F(z;v)$ for any $(z;v)\in T^{1,0}D$ and $\lambda\in\mathbb{C}$.
\end{definition}

\begin{remark}
It is important to ask for the smoothness of $G$ only on $T^{1,0}D\setminus\{\mbox{zero section}\}$ since $G$ is smooth (in fact $C^2$ is enough) on the whole $T^{1,0}D$ if and only if $F$ is a Hermitian metric (cf. Lemma 2.3.1 in \cite{AP}).
\end{remark}

Setting $z^i=x^i+\sqrt{-1}x^{m+i}$ and
$v^i=u^i+\sqrt{-1}u^{m+i}$ for $i=1, \cdots, m.$ Then $(x;u)=(x^1, \cdots, x^{2m}; u^1, \cdots, u^{2m})$ can be used as real coordinates on the real tangent bundle $TD$ of $D$ (considered as an open set in $\mathbb{R}^{2m}$). Note that $T^{1,0}D\cong TD$.

\begin{definition}
A complex Finsler metric $F$ is called strongly pseudoconvex if the $m$-by-$m$ Levi matrix $(G_{i\overline{j}})$ is positive definite on $T^{1,0}D\setminus\{\mbox{zero section}\}$, here
\begin{equation}
G_{i\overline{j}}:=\frac{\partial^2 G}{\partial v^i\partial\overline{v^j}},\quad i, j=1, \cdots, m.\label{spc}
\end{equation}
 If in particular, a strongly pseudoconvex complex Finsler metric $F$ is independent of $z\in D$ for any $v\in T_{z}^{1,0}D\cong \mathbb{C}^m$, then $F$ is called a complex Minkowski metric on $D$.

A complex Finsler metric $F$ is called strongly convex if the real $2m$-by-$2m$ Hessian matrix
$(G_{ab})$ is positive definite on $TD\setminus\{\mbox{zero section}\}$, here
$$
G_{ab}:=\frac{\partial^2G}{\partial u^a\partial u^b},\quad a, b=1, \cdots, 2m.
$$
\end{definition}
Note that $F$ is a strongly pseudoconvex (resp. strongly convex) complex Finsler metric on $D$ is equivalent to requiring that all the $F$-indicatries
$$
I_F(z)=\left\{v\in T_z^{1,0}D|\,F(z;v)<1\right\},\quad z\in D
$$
are strongly pseudoconvex (resp. strongly convex).

\begin{example} Let $D$ be a domain in $\mathbb{C}^m$ and
$ds^2=g_{i\overline{j}}(z)dz^id\overline{z^j}$
be a Hermitian metric (which is always assumed to be $C^\infty$) on $D$ . Then
$$
F(z;v)=\sqrt{g_{i\overline{j}}(z)v^i\overline{v^j}}
$$
is a strongly pseudoconvex complex Finsler metric, which is called Hermitian quadratic.
\end{example}
\begin{example}\cite{Zh2}
Let $D$ be a domain in $\mathbb{C}^m$. For any fixed $t\in(0,+\infty)$ and any fixed integer $k\geq 2$, define
$$
F_{t,k}(z;v)=\sqrt{|v^1|^2+\cdots+|v^m|^2+t\sqrt[k]{|v^1|^{2k}+\cdots+|v^m|^{2k}}},\quad \forall v\in T_z^{1,0}D.
$$
Then $F_{t,k}$ is a strongly convex complex Minkowski metric, which is not Hermitian quadratic.
\end{example}

Associated to a strongly pseudoconvex complex Finsler metric $F:T^{1,0}D\rightarrow [0,+\infty)$, the Chern-Finsler nonlinear connection coefficients are given by \cite{AP}
$$\varGamma_{;l}^i=G^{\overline{s}i}\frac{\partial^2G}{\partial \overline{v^s}\partial z^l},$$
where $(G^{\overline{s}i})$ is the inverse matrix of $(G_{i\overline{s}})$,
and the complex Berwald nonlinear connection coefficients are given by
$$
\mathbb{G}_{\;l}^i=\frac{\partial}{\partial v^l}\left(\frac{1}{2}\varGamma_{;l}^i v^l\right).
$$
The horizontal Chern-Finsler connection coefficients are given by
$$
\varGamma_{j;l}^i=\frac{\partial \varGamma_{;l}^i}{\partial v^j}
$$
and the complex Berwald connection coefficients are given by
$$
\mathbb{G}_{jl}^i=\frac{\partial \mathbb{G}_{\;l}^i}{\partial v^j}.
$$
A strongly pseudoconvex complex Finsler metric $F:T^{1,0}D\rightarrow [0,+\infty)$ is called a K\"ahler-Finsler metric in \cite{AP} and \cite{CS} if
$$
\varGamma_{j;l}^i=\varGamma_{l;j}^i;
$$
called a complex Berwald metric if the horizontal Chern-Finsler connection coefficients $\varGamma_{j;l}^i$ are independent of $v=(v^1,\cdots,v^m)$ (\cite{A-1}); called a weakly complex Berwald metric if the complex Berwald connection coefficients $\mathbb{G}_{jl}^i$ are independent of $v=(v^1,\cdots,v^m)$ (\cite{Zh0}); called a K\"ahler-Berwald metric if $F$ is both a K\"ahler-Finsler metric and a complex Berwald metric.
It was shown in \cite{Zh2} that the metric defined by \eqref{Ftk} is an $\mbox{Aut}(P_m)$-invariant strongly convex K\"ahler-Berwald metric on $P_m$, which reduces to the usual Bergman metric on $P_m$ whenever $t=0$, and it is non-Hermitian quadratic for any fixed $t\in (0,+\infty)$ and any fixed integer $k\geq 2$.

\section{Some lemmas} \label{se2}

The following Lemma \ref{l-2.1} and Lemma \ref{l-2.2} actually generalize  Lemma 4 and Lemma 5 in K. H. Look \cite{Look2} respectively in the sense that $P_m$ and $P_n$ need not have the same dimensions.
The proofs go essentially the same lines as that of Lemma 4 and Lemma 5 in \cite{Look2}.  We enclose here the proofs for completeness.
\begin{lemma}\label{l-2.1}
Let $w=f(z)$ be an holomorphic mapping carrying $P_{m}$ into $P_{n}$ and carrying the origin $0\in\mathbb{C}^m$ to the origin $0 \in\mathbb{C}^n$. Let the developments of the mapping functions $f_{l}(z)(l=1,\cdots,n)$ be
\begin{eqnarray*}
w^{1}&=&f_{1}(z)=a_{11} z^{1}+\cdots+a_{m 1} z^{m}+\mbox{terms of higher power},\\
&\vdots&\\
w^{n}&=&f_{n}(z)=a_{1 n} z^{1}+\cdots+a_{m n} z^{m}+\mbox{terms of higher power}.
\end{eqnarray*}
Then the linear terms of this transformation
\begin{eqnarray*}
	w^{1}&=&a_{11} z^{1}+\cdots+a_{m 1} z^{m}, \\
	   &\vdots&\\
	w^{n}&=&a_{1 n} z^{1}+\cdots+a_{mn} z^{m}
\end{eqnarray*}
form a holomorphic mapping from $P_{m}$ into $P_n$.
\end{lemma}

\begin{proof}
Let $\lambda\in\mathbb{C}$ and $|\lambda| \leq 1$.~We take an arbitrarily fixed interior point $z$ of $P_m$, then we have
$$
\left|\lambda z^{1}\right|<1, \cdots, \left|\lambda z^{m}\right|<1.
$$
We consider the functions
$$
f_{l}(\lambda z), \quad l=1, \cdots, n,
$$
which are holomorphic functions of $\lambda$ in the closed disk $\overline{\triangle}$.~By Schwarz lemma of one complex variable, for each $f_{l}(\lambda z)$ we have
\begin{equation}
\left|\frac{d f_{l}(\lambda z)}{d \lambda}\right|_{\lambda=0} \leq 1, \quad l=1, \cdots, n.
\label{lam3.1}
\end{equation}
However, it can be easily verified that
$$
\left|\frac{d f_{l}(\lambda z)}{d \lambda}\right|_{\lambda=0}<1, \quad l=1, \cdots, n.
$$
Indeed, if the equality holds in \eqref{lam3.1}, then by the Schwarz lemma of one complex variable we have $f_{l}(\lambda{z})=e^{i \theta} \lambda$ and $\left|f_{l}(z)\right|=1$ when we set $\lambda=1$.~But by our hypothesis, $z$ is an interior point of $P_{m}$ and the mapping is an inner holomorphic mapping.~We have $\left|f_{l}(\lambda)\right|<1$ whenever $z \in P_{m}$.~These two results are contradictory to each other.~Hence we have
$$
\left|\frac{d f_{l}(\lambda z)}{d \lambda}\right|_{\lambda=0}<1, \quad l=1, \cdots, n,
$$
i.e.,
$$
\left|a_{1l} z^{1}+\cdots+a_{m l} z^{m}\right|<1, \quad l=1, \cdots, n.
$$
This shows that the linear transformation
$$
w^{l}=a_{1l} z^{1}+\cdots+a_{ml} z^{m}, \quad l=1, \cdots, n
$$
carries $P_{m}$ into $P_n$. This completes the proof.
\end{proof}

\begin{lemma}\label{l-2.2}
The necessary and sufficient condition of the linear transformation
\begin{equation}
w^{l}=a_{1l} z^{1}+\cdots+a_{ml} z^{m}, \quad l=1, \cdots, n \label{2.1}
\end{equation}
carries $P_{m}$ into $P_n$, is that
\begin{equation}
|a_{1l}|+\cdots+|a_{ml}|\leq1, \quad l=1, \cdots, n. \label{2.2}
\end{equation}
\end{lemma}

\begin{proof}
The sufficiency is obvious, since for any point $z\in P_{m}$
\begin{eqnarray*}
|w^{l}|
&\leq&\left|a_{1 l}\right| \cdot\left|z^{1}\right|+\cdots+\left|a_{m l}\right| \cdot\left|z^{m}\right| \\
&<&\left|a_{1l}\right|+\cdots+\left|a_{ml}\right| \leq 1, \quad l=1, \cdots, n.
\end{eqnarray*}

For the proof of the necessity, we denote by $\overline{P_m}$ the closure of $P_{m}$. Since the linear mapping \eqref{2.1} carries $P_{m}$ into $P_{n}$, for any point $z \in \overline{P_m}$ we have
\begin{equation}
\left|a_{1l} z^{1}+\cdots+a_{m l} z^{m}\right| \leq 1, \quad l=1, \cdots, n. \label{2.3}
\end{equation}

Let $a_{i l}=\left|a_{i l}\right| e^{\sqrt{-1} \theta_{il}}$ for $i=1, \cdots, m$ and $l=1, \cdots, n$, where $\theta_{il}$ are real numbers. When we set
$z^{1}=e^{-\sqrt{-1} \theta_{1 l}}, \cdots, z^{m}=e^{-\sqrt{-1} \theta_{m l}}$ into \eqref{2.3}, we obtain
$$
\left|a_{1 l}\right|+\cdots+\left|a_{m l}\right| \leq 1.
$$
This completes the proof.
\end{proof}
\section{Schwarz lemma on polydiscs endowed with K\"ahler-Berwald metrics} \label{se3}

In this section, we shall give another proof of the following theorem which was first proved in  \cite{Zh2}, and then establish a Schwarz lemma for holomorphic mappings from $P_m$ into $P_n$, where $P_m$ and $P_n$ endowed with $\mbox{Aut}(P_m)$-invariant K\"ahler-Berwald metric $F_{t,k}$ and $\mbox{Aut}(P_n)$-invariant K\"ahler-Berwald metric $\tilde{F}_{\tilde{t},\tilde{k}}$, respectively.
\begin{theorem}\label{Z}
Let $P_m$ be the  unit polydisc in $\mathbb{C}^m.$ Then
\begin{equation}
F_{t,k}(z;v)=\frac{1}{\sqrt{1+t}}\sqrt{\sum_{l=1}^m\frac{|v^l|^2}{(1-|z^l|^2)^2}+t\sqrt[k]{\sum_{l=1}^m\frac{|v^l|^{2k}}{(1-|z^l|^2)^{2k}}}},\quad z\in P_m,v\in T_z^{1,0}P_m\label{ftk}
\end{equation}
is an $\mbox{Aut}(P_m)$-invariant strongly convex K\"ahler-Berwald metric for any fixed $t\in[0,+\infty)$ and any fixed integer $k\geq 2$.
\end{theorem}
\begin{remark}
The factor $\frac{1}{\sqrt{1+t}}$ is put here so that $F_{t,k}$ reduces to the well-known Poincar$\acute{\mbox{e}}$ metric on the  unit disk $\triangle$ in $\mathbb{C}$ when $m=1$. For $m\geq 2$ and $t=0$, $F_{t,k}$ reduces to the Bergman metric on the polydisc $P_m$ which is an $\mbox{Aut}(P_m)$-invairant K\"ahler metric, hence Hermitian quadratic. For any $m\geq 2$, fixed $t\in(0,+\infty)$, and fixed $k\in\mathbb{N},k\geq 2$, however, $F_{t,k}$ is an $\mbox{Aut}(P_m)$-invariant complete K\"ahler-Berwald metric which is non-Hermitian quadratic.
\end{remark}
\begin{proof}The proof given here essentially uses the transitive action of $\mbox{Aut}(P_m)$ on $P_m$.

That \eqref{ftk} is a strongly convex K\"ahler-Berwald metric as a direct consequence of Theorem 1.1 in \cite{Zh2}. To prove that \eqref{ftk} is $\mbox{Aut}(P_m)$-invariant. It suffices to show that \eqref{ftk} can be recovered by assigning a strongly convex complex norm $\varphi_{t,k}(v)$ on the tangent space $T_0^{1,0}P_m\cong\mathbb{C}^m$ at the origin satisfying
\begin{equation}
 \varphi_{t,k}\big(g_\ast(v)\big)=\varphi_{t,k}(v),\quad\forall g\in\mbox{Iso}(P_m),\label{ivn}
\end{equation}
where $\mbox{Iso}(P_m)$ is the isotropy subgroup at the origin and $g_\ast$ denotes the differential of $g$ at the origin.

Indeed, if we define
\begin{eqnarray}
\varphi_{t,k}(v):=\frac{1}{\sqrt{1+t}}\sqrt{\sum_{l=1}^m|v^l|^2+t\sqrt[k]{\sum_{l=1}^m|v^l|^{2k}}},\quad v\in T_0^{1,0}P_m,\label{nftk}
\end{eqnarray}
then by Theorem 1.1 in \cite{Zh2}, $\varphi_{t,k}(v)$ is a strongly convex complex norm on $T_0^{1,0}P_m$.

By Proposition 3 of Chapter 5 in \cite{Na}, the automorphism group $\mbox{Aut}(P_m)$ consists holomorphic functions of the following form:
\begin{equation}
f_{z_0}(z)
=\left(e^{i\theta_{1}}\frac{z^{\sigma(1)}-z_0^{1}}{1-\overline{z_0}^{1}z^{\sigma(1)}},
\cdots,
e^{i\theta_{m}}\frac{z^{\sigma(m)}-z_0^{m}}{1-\overline{z_0}^{m}z^{\sigma(m)}}\right),\quad z\in P_m,
\end{equation}
where $z_0=(z_0^1,\cdots,z_0^m)\in P_m$, $\theta_1,\cdots,\theta_m\in\mathbb{R}$ and $\sigma:\{1,\cdots,m\}\rightarrow\{1,\cdots,m\}$ is a permutation.
The isotropy subgroup $\mbox{Iso}(P_m)$  at the origin is given by
$$\mbox{Iso}(P_m)=\left\{g(z)=\left(e^{i\theta_1}z^{\sigma(1)},\cdots, e^{i\theta_m}z^{\sigma(m)}\right)\Big|\theta_1,\cdots,\theta_m\in\mathbb{R}\right\},$$
which together with  \eqref{nftk} immediately implies \eqref{ivn}.

Next we show that
\begin{equation}
F_{t,k}(z_0;v)=\varphi_{t,k}((h_{z_0})_\ast(v)),\label{hass}
\end{equation}
where
\begin{equation}
h_{z_0}(z)=\left(\frac{z^1-z_0^{1}}{1-\overline{z_0^1}z^1},
\cdots,\frac{z^m-z_0^{m}}{1-\overline{z_0^m}z^m}\right),\quad z\in P_m.\label{ha}
\end{equation}
Notice $h_{z_0}\in \mbox{Aut}(P_m)$ with $h_{z_0}(z_0)=0$,  and for each $v\in T_{z_0}^{1,0}P_m$, we have
\begin{equation}
(h_{z_0})_\ast(v)=dh_{z_0}(v)=\left(\frac{v^1}{1-|z_0^1|^2},\cdots,\frac{v^m}{1-|z_0^m|^2}\right)\in T_0^{1,0}P_m.\label{has}
\end{equation}
Substituting \eqref{has} into \eqref{nftk}, we get
\begin{eqnarray}
\varphi_{t,k}\big((h_{z_0})_\ast(v)\big)
&=&\frac{1}{\sqrt{1+t}}\sqrt{\sum_{l=1}^m\frac{|v^l|^2}{(1-|z_0^l|^2)^2}+t\sqrt[k]{\sum_{l=1}^m\frac{|v^l|^{2k}}{(1-|z_0^l|^2)^{2k}}}},\label{have}
\end{eqnarray}
this is, \eqref{hass} holds.

Finally, we show that $F_{t,k}(z;v)$, thus obtained, is indeed $\mbox{Aut}(P_m)$-invariant.
For any $\phi\in\mbox{Aut}(P_m)$, we denote $b=\phi(z_0)$. Then $ h_b\circ \phi\circ h_{z_0}^{-1}(0)=h_b\circ \phi(z_0)=h_b(b)=0$ and $h_b\circ \phi\circ h_{z_0}^{-1}\in\mbox{Iso}(P_m)$. Thus by \eqref{ivn} and \eqref{have}, we have
\begin{eqnarray*}
	F_{t,k}\big(\phi(z_0);\phi_\ast (v)\big)
	&=&\varphi_{t,k}\big((h_b)_\ast\circ \phi_\ast(v)\big)\quad\mbox{by}\;\eqref{hass}\\
	&=&\varphi_{t,k}\big((h_b\circ \phi\circ h_{z_0}^{-1})_\ast\circ (h_{z_0})_\ast(v)\big)\\
	&=&\varphi_{t,k}\big((h_{z_0})_\ast(v)\big)\quad \mbox{by}\;\eqref{ivn}\\
	&=&F_{t,k}(z_0;v).\quad\mbox{by}\;\eqref{hass}
\end{eqnarray*}
Since $z_0$ is an arbitrary fixed point in $P_m$, changing $z_0$ to $z$ if necessary we complete the assertion that $F_{t,k}$ is $ \mbox{Aut}(P_m)$-invariant.
\end{proof}

\begin{theorem}\label{th3}
Let $P_{m}$ be the  unit polydisc in $\mathbb{C}^m$ endowed with the $\mbox{Aut}(P_m)$-invariant K\"ahler-Berwald metric $F_{t, k}$, and $P_n$ be the  unit polydisc in $\mathbb{C}^n$ endowed with the $\mbox{Aut}(P_n)$-invariant K\"ahler-Berwald metric $\tilde{F}_{\tilde{t},\tilde{k}}.$ Then for any holomorphic mapping $f:P_m\rightarrow P_n$, we have
\begin{equation}
\big(f^\ast \tilde{F}_{\tilde{t},\tilde{k}}^2\big)(z;v)\leq \frac{n+\tilde{t}\sqrt[\tilde{k}]{n}}{1+\tilde{t}}F_{t,k}^2(z;v),\quad \forall z\in P_m, v\in T_z^{1,0}P_m.\label{t-1}
\end{equation}
The constant $\frac{n+\tilde{t}\sqrt[\tilde{k}]{n}}{1+\tilde{t}}$ in \eqref{t-1} is optimal for any fixed $t,\tilde{t}\in[0,+\infty)$ and any fixed integers $k,\tilde{k}\geq 2$.
\end{theorem}

\begin{proof}

Let $z_0=(z_0^1,\cdots,z_0^m)$ be an arbitrary fixed point in $P_m$ and $f:P_m\rightarrow P_n$ be an arbitrary holomorphic mapping. Denote $w_0=(w_0^1,\cdots,w_0^n)=f(z_0)\in P_n$.

Define
\begin{equation}
h_{z_0}(z)
:=\left(\frac{z^1-z_0^1}{1-\overline{z_0^1}z^1},
\cdots,
\frac{z^m-z_0^m}{1-\overline{z_0^m}z^m}\right),\quad z\in P_m
\label{3.3}
\end{equation}
 and
\begin{equation}
\tilde{h}_{w_0}(w)
:=\left(\frac{w^1-w_0^1}{1-\overline{w_0^1}w^1},
\cdots,
\frac{w^n-w_0^n}{1-\overline{w_0^n}w^n}\right),\quad w\in P_n.
\label{3.4}
\end{equation}

Then $h_{z_0}\in\mbox{Aut}\left(P_{m}\right)$ and $\tilde{h}_{w_0}\in\mbox{Aut}\left(P_{n}\right)$. By \eqref{3.3} and \eqref{3.4} we have
\begin{equation}
h_{z_0}(z_0)=0\in P_m\quad \text{and} \quad \tilde{h}_{w_0}(w_0)=0\in P_n,\label{hh}
\end{equation}
which together with $w_0=f(z_0)$ implies that $g:=\tilde{h}_{w_0}\circ f\circ h_{z_0}^{-1}$ is a holomorphic mapping from $P_{m}$ into $P_{n}$ satisfying $g(0)=0$.
By definition, we have
\begin{equation}
\big(f^\ast\tilde{F}_{\tilde{t},\tilde{k}}^{2}\big)(z_0;v)=\tilde{F}_{\tilde{t},\tilde{k}}^{2}\big(f(z_0); f_\ast(v)\big), \label{ptf}
\end{equation}
where $f_\ast$ denotes the differential of $f$ at the point $z_0$.
On the other hand, $\tilde{F}_{\tilde{t},\tilde{k}}$ is $\mbox{Aut}\left(P_{n}\right)$-invariant and $\tilde{h}_{w_0}\in \mbox{Aut}(P_n)$, thus
\begin{eqnarray}
\tilde{F}_{\tilde{t},\tilde{k}}^{2}\big(f(z_0);f_\ast(v)\big)
&=&\tilde{F}_{\tilde{t},\tilde{k}}^{2}\Big(\tilde{h}_{w_0}\big(f(z_0)\big); \big(\tilde{h}_{w_0}\big)_\ast\big(f_\ast(v)\big)\Big)\nonumber\\
&=&\tilde{F}_{\tilde{t},\tilde{k}}^{2}\Big(\big(\tilde{h}_{w_0} \circ f \circ h_{z_0}^{-1}\big)\big(h_{z_0}(z_0)\big);\big(\tilde{h}_{w_0} \circ f \circ h_{z_0}^{-1}\big)_\ast\big((h_{z_0})_\ast(v)\big)\Big)\nonumber\\
&=&\tilde{F}_{\tilde{t},\tilde{k}}^{2}\Big(g(0); g_\ast\big((h_{z_0})_\ast(v)\big)\Big)\nonumber\\
&=&\tilde{F}_{\tilde{t},\tilde{k}}^2\Big(0 ; g_\ast\big((h_{z_0})_\ast(v)\big)\Big)\label{a-a}
\end{eqnarray}
for any $v\in T_{z_0}^{1,0}P_m$.

Now denote $u:=(h_{z_0})_\ast(v)\in T_{0}^{1,0}P_m$ and $g(z)=(g_{1}(z), \cdots, g_{n}(z))$. Then by Lemma \ref{l-2.1} and Lemma \ref{l-2.2}, we have
\begin{eqnarray}
\tilde{F}_{\tilde{t},\tilde{k}}^2\Big(0;g_\ast\big((h_{z_0})_\ast(v)\big)\Big)&=&\tilde{F}_{\tilde{t},\tilde{k}}^2\big(0;g_\ast(u)\big)\nonumber\\
&=&\frac{1}{1+\tilde{t}}\left\{\sum_{l=1}^{n} \left|\sum_{j=1}^{m}u^{j}a_{jl}\right|^{2}+\tilde{t} \sqrt[\tilde{k}]{\sum_{l=1}^{n} \left|\sum_{j=1}^{m}u^{j}a_{jl}\right|^{2\tilde{k}}}\right\}\nonumber\\
&\leq& \frac{1}{1+\tilde{t}}\left\{\sum_{l=1}^{n}\left(\sum_{j=1}^{m}|u^{j}||a_{jl}|\right)^{2}+\tilde{t}\sqrt[\tilde{k}]{\sum_{l=1}^{n} \left(\sum_{j=1}^{m}|u^{j}||a_{jl}|\right)^{2\tilde{k}}}\right\}\nonumber\\
&\leq& \frac{1}{1+\tilde{t}}\left\{\sum_{l=1}^{n}\left(\max_{1\leq j\leq  m} \big\{|u^j|\big\}\sum_{j=1}^{m}|a_{jl}|\right)^{2}+\tilde{t}\sqrt[\tilde{k}]{\sum_{l=1}^{n}\left(\max_{1\leq j\leq m} \big\{|u^{j}|\big\} \sum_{j=1}^{m}|a_{jl}|\right)^{2\tilde{k}}}\right\}\nonumber\\
&\leq& \frac{1}{1+\tilde{t}}\left\{\sum_{l=1}^{n}\left(\max_{1\leq j\leq m} \big\{|u^{j}|\big\}\right)^{2}+t\sqrt[\tilde{k}]{\sum_{l=1}^{n} \Big(\max_{1\leq j\leq m} \big\{|u^{j}|\big\}\Big)^{2\tilde{k}}}\right\}\nonumber\\
&=&\frac{n+\tilde{t}\sqrt[\tilde{k}]{n}}{1+\tilde{t}}\max_{1\leq j\leq m} \left\{|u^{j}|^2\right\}\label{a-b}.
\end{eqnarray}
On the other hand, since $F_{t,k}$ is $\mbox{Aut}(P_m)$-invariant and $h_{z_0}\in\mbox{Aut}(P_m)$, we have
\begin{eqnarray}
F_{t,k}^2(z_0;v)&=&F_{t,k}^2\big(h_{z_0}(z_0);(h_{z_0})_\ast(v)\big)\nonumber\\
&=&F_{t,k}^2\big(0;(h_{z_0})_\ast(v)\big)\nonumber\\
&=&F_{t, k}^{2}(0; u)\nonumber\\
&=&\frac{1}{1+t}\left\{\sum_{j=1}^{m}|u^{j}|^2+t \sqrt[k]{\sum_{j=1}^{m} |u^j|^{2k}}\right\}\nonumber\\
&\geq&\frac{1}{1+t}\left\{\max_{1\leq j\leq m} \big\{|u^j|^2\big\}+t \sqrt[k]{\max_{1\leq j\leq m} \big\{|u^{j}|^{2k}\big\}}\right\}\nonumber\\
&=&\max_{1\leq j\leq m} \big\{|u^{j}|^2\big\}.\label{a-c}
\end{eqnarray}
By \eqref{a-a}-\eqref{a-c}, we obtain
$$\big(f^\ast \tilde{F}_{\tilde{t},\tilde{k}}^{2}\big)(z_0; v)\leq\frac{n+\tilde{t}\sqrt[\tilde{k}]{n}}{1+\tilde{t}}F_{t, k}^{2}(z_0; v),\quad \forall v\in T_{z_0}^{1,0}P_m.$$
Since $z_0$ is an arbitrary fixed point in $P_m$, changing $z_0$ to $z$ if necessary we obtain \eqref{t-1}.


To see that the constant $\frac{n+\tilde{t}\sqrt[\tilde{k}]{n}}{1+\tilde{t}}$ is optimal, let's take $w=f_0(z)=\big(z^1,\cdots,z^1\big)$ (that is, $w^l=z^1$ for $l=1,\cdots,n$) and $v_0=(1,0,\cdots,0)\in T_0^{1,0}P_m$. Then
$$(f_0)_\ast(v_0)=(1,\cdots,1),$$
where $(f_0)_\ast$ denotes the differential of $f_0$ at the point $z_0=0$. It is easy to check that
$$\big(f_0^\ast\tilde{F}_{\tilde{t},\tilde{k}}^2\big)(z_0;v_0)=\frac{n+\tilde{t}\sqrt[\tilde{k}]{n}}{1+\tilde{t}},\quad F_{t,k}^2(z_0;v_0)=1.$$
Consequently
$$
\big(f_0^\ast\tilde{F}_{\tilde{t},\tilde{k}}^2\big)(z_0;v_0)=\frac{n+\tilde{t}\sqrt[\tilde{k}]{n}}{1+\tilde{t}}F_{t,k}^2(z_0;v_0).
$$
This completes the proof.
\end{proof}
\begin{corollary}
	Let $P_m$ be the  unit polydisc in $\mathbb{C}^m$ endowed with the Bergman metric $F_0$, and $P_n$ be the unit polydisc in $\mathbb{C}^n$ endowed with the $\mbox{Aut}(P_n)$-invariant K\"ahler-Berwald metric $\tilde{F}_{\tilde{t},\tilde{k}}.$ Then for any holomorphic mapping $f: P_m\rightarrow P_n$, we have
	\begin{equation}
		\big(f^\ast \tilde{F}_{\tilde{t},\tilde{k}}^2\big)(z;v)\leq \frac{n+\tilde{t}\sqrt[\tilde{k}]{n}}{1+\tilde{t}}F_0^2(z;v),\quad \forall z\in P_m, v\in T_z^{1,0}P_m.\label{t-2}
	\end{equation}
\end{corollary}

\begin{corollary}
	Let $P_m$ be the  unit polydisc in $\mathbb{C}^m$ endowed with the $\mbox{Aut}(P_m)$-invariant K\"ahler-Berwald metric $F_{t,k}$, and $P_n$ be the unit polydisc in $\mathbb{C}^n$ endowed with the Bergman metric $\tilde{F}_0.$ Then
	for any holomorphic mapping $f:P_m\rightarrow P_n$, we have
	\begin{equation}
		\big(f^\ast \tilde{F}_0^2\big)(z;v)\leq nF_{t,k}^2(z;v),\quad \forall z\in P_m, v\in T_z^{1,0}P_m.\label{t-3}
	\end{equation}
\end{corollary}

\begin{corollary}\label{cor-3}
	Let $P_m$ be the  unit polydisc in $\mathbb{C}^m$ endowed with the Bergman metric $F_0$, and $P_n$ be the unit polydisc in $\mathbb{C}^n$ endowed with the Bergman metric $\tilde{F}_0.$ Then
	for any holomorphic mapping $f:P_m\rightarrow P_n$, we have
	\begin{equation}
		\big(f^\ast \tilde{F}_0^2\big)(z;v)\leq nF_0^2(z;v),\quad \forall z\in P_m, v\in T_z^{1,0}P_m.\label{t-4}
	\end{equation}
\end{corollary}
\begin{remark}
	Corollary \ref{cor-3} reduces to Theorem \ref{th1} whenever $m=n$.
\end{remark}

When $m=n$ in the Theorem \ref{th3}, we have the following Proposition \ref{prop3}.
\begin{proposition}\label{prop3}
	Let $P_{m}$ be the  unit polydisc in $\mathbb{C}^m$ endowed with the $\mbox{Aut}(P_m)$-invariant K\"ahler-Berwald metric $F_{t, k}$, and $f$ be a holomorphic mapping  from $P_m$ into itself. If there exists a point $z_0 \in P_m$ such that
	\begin{eqnarray}
		\big(f^\ast F^2_{t,k}\big)(z_0;v) \geq F^2_{t,k}(z_0; v),\quad \forall z_0\in P_m, v\in T_{z_0}^{1,0}P_m,\label{th3-4}
	\end{eqnarray}
	then  $f \in \mbox{Aut}(P_m)$.
\end{proposition}

\begin{proof}
Take $\varphi \in \mbox{Aut}(P_m)$ such that $\varphi(z_0)=f(z_0)$, which is possible since $P_m$ is homogeneous. Then $ h:=\varphi^{-1} \circ f $ is a holomorphic map from $P_m$ into itself and keeps $z_0$ fixed. Note that $\varphi^{-1}\in \mbox{Aut}(P_m)$, hence it is also an isometry of $F_{t,k}$, which together with  the condition \eqref{th3-4} implies
\begin{eqnarray*}
	(h^{*}F^2_{t,k})(z_0; v)=\left((\varphi^{-1})^\ast\circ f^\ast\right) F_{t,k}^2(z_0;v)=f^\ast F_{t,k}^2(z_0;v) \geq F^2_{t,k}(z_0; v), \quad\forall v\in T_{z_0}^{1,0}P_m.
\end{eqnarray*}

Now let $\mu_i(i=1,\cdots,m)$ be an eigenvalue of the matrix $(h_{\ast })_{z_0}$, and  $v_i$ be an eigenvector corresponding to $\mu_i$,  that is,  $( h_{\ast })_{z_0}v_i=\mu_i v_i$. Then
\begin{eqnarray*}
(h^{*}F^2_{t,k})(z_0; v_i)= F^2_{t,k}(z_0; \mu_i v_i)= |\mu_i|^2F^2_{t,k}(z_0;  v_i)\geq F^2_{t,k}(z_0; v_i),\quad i=1,\cdots,m,
\end{eqnarray*}
thus we have $|\mu_i|\geq1$, and conseqently $\left|\det\left(( h_{\ast })_{z_0}\right)\right|=|\mu_1\cdots\mu_m|\geq 1.$ By a well known theorem of H. Cartan and Carath\'{e}odory (e.g., in Chapter 1 of \cite{BM}) it follows that $h \in \mbox{Aut}(P_m)$, from which it follows that $f \in \mbox{Aut}(P_m)$. This completes the proof.
\end{proof}

\begin{proposition}\label{prop1}
Let $f:P_m\rightarrow P_n$ be a holomorphic mapping satisfying $f(0)=0.$ Then for any fixed $l\in\{1, 2, \cdots, m\}$ and
$0\neq v=\big(\underbrace{0, \cdots, 0}_{l-1}, v^l, 0, \cdots, 0\big)\in T^{1,0}_0P_m$,
\begin{eqnarray}
\big(f^\ast\tilde{F}_{\tilde{t},\tilde{k}}^2\big)(0;v)=\frac{n+\tilde{t}\sqrt[\tilde{k}]{n}}{1+\tilde{t}}F_{t,k}^2(0;v) \label{b-c}
\end{eqnarray}	
holds if and only if
\begin{eqnarray}
f_j\big(\underbrace{0, \cdots, 0}_{l-1}, z^l, 0, \cdots, 0\big)=e^{\sqrt{-1}\theta_j}z^l, \quad \theta_j\in \mathbb{R}, j=1, \cdots, n.	\label{b-d}
\end{eqnarray}	
\end{proposition}

\begin{proof}
First we prove the necessity. By definition,
$$
\big(f^\ast\tilde{F}_{\tilde{t},\tilde{k}}^2\big)(0;v)
=\frac{1}{1+\tilde{t}}\left\{\sum_{j=1}^{n} \left|\sum_{i=1}^{m}\frac{\partial f_j}{\partial z^i }(0)v^{i}\right|^{2}+\tilde{t}\sqrt[\tilde{k}]{\sum_{j=1}^{n} \left|\sum_{i=1}^{m}\frac{\partial f_j}{\partial z^i }(0)v^{i}  \right|^{2\tilde{k}}}\right\},\quad \forall v\in T_0^{1,0}P_m.
$$
 Thus for any fixed $l\in\{1, 2, \cdots, m\}$ and
$0\neq v=\big(\underbrace{0, \cdots, 0}_{l-1}, v^l, 0, \cdots, 0\big)\in T^{1,0}_0P_m$, we have
\begin{eqnarray}
\big(f^\ast\tilde{F}_{\tilde{t},\tilde{k}}^2\big)(0;v)
&=&\frac{1}{1+\tilde{t}}\left\{\sum_{j=1}^{n} \left|\frac{\partial f_j}{\partial z^l }(0)v^{l}\right|^{2}+\tilde{t}\sqrt[\tilde{k}]{\sum_{j=1}^{n} \left|\frac{\partial f_j}{\partial z^l }(0)v^{l}  \right|^{2\tilde{k}}}\right\}.\label{4-2-a}
\end{eqnarray}

By Lemma \ref{l-2.2}, we have
$$
\left|\frac{\partial f_j}{\partial z^l }(0)\right|\leq 1
$$
for any  $j=1,\cdots,n$, which together with \eqref{4-2-a} implies that
\begin{eqnarray}
\big(f^\ast\tilde{F}_{\tilde{t},\tilde{k}}^2\big)(0;v)
&\leq& \frac{1}{1+\tilde{t}}\left\{\sum_{j=1}^{n}|v^{l}|^{2}+\tilde{t}\sqrt[\tilde{k}]{\sum_{j=1}^{n} |v^{l}|^{2\tilde{k}}}\right\}\nonumber\\
&=&\frac{n+\tilde{t}\sqrt[\tilde{k}]{n}}{1+\tilde{t}}|v^{l}|^2\nonumber\\
&=&\frac{n+\tilde{t}\sqrt[\tilde{k}]{n}}{1+\tilde{t}}F_{t,k}^2(0;v)\label{a-r}.
\end{eqnarray}
Therefore if the equality in \eqref{b-c} holds, then it necessary that
$$\left|\frac{\partial f_j}{\partial z^l }(0)\right|=1,\quad j=1,\cdots,n.$$
So that if we define holomorphic mappings
$$
g_j:\triangle \longrightarrow \triangle,\quad g_j\big(\zeta\big):=f_j\big(\underbrace{0, \cdots, 0}_{l-1}, \zeta, 0, \cdots, 0\big),\quad j=1,\cdots,n.
$$
Then $g_j(0)=0,$ and $ |g'_j(0)|=\left|\frac{\partial f_j}{\partial z^l }(0)\right|=1.$ Thus by the classical Schwarz lemma of one complex variable, we have $g_j(\zeta)=e^{\sqrt{-1}\theta_j}\zeta$, so that
$$
f_j\big(\underbrace{0, \cdots, 0}_{l-1}, z^l, 0, \cdots, 0\big)=g_j(z^l)=e^{\sqrt{-1}\theta_j}z^l,\quad \theta_j\in \mathbb{R}, j=1, \cdots, n.
$$

Now we prove the sufficiency. If $f_j\big(\underbrace{0, \cdots, 0}_{l-1}, z^l, 0, \cdots, 0\big)=e^{\sqrt{-1}\theta_j}z^l$ for  $\theta_j\in \mathbb{R}$ and $j=1, \cdots, n$. Taking $v_0=\big(\underbrace{0, \cdots, 0}_{l-1}, 1, 0, \cdots, 0\big)\in T_0^{1,0}P_m$,  we have $f_\ast(v_0)=\big(e^{\sqrt{-1}\theta_1},\cdots,e^{\sqrt{-1}\theta_n}\big)$ and
$$
\big(f^\ast\tilde{F}_{\tilde{t},\tilde{k}}^2\big)(0;v_0)
=\frac{1}{1+\tilde{t}}\left\{\sum_{j=1}^{n}\big|e^{\sqrt{-1}\theta_j}\big|^{2}+\tilde{t}\sqrt[\tilde{k}]{\sum_{j=1}^{n} \big|e^{\sqrt{-1}\theta_j}\big|^{2\tilde{k}}}\right\}
=\frac{n+\tilde{t}\sqrt[\tilde{k}]{n}}{1+\tilde{t}}, \quad F_{t,k}^2(0;v_0)=1.
$$
Consequently
$$\big(f^\ast\tilde{F}_{\tilde{t},\tilde{k}}^2\big)(0;v_0)=\frac{n+\tilde{t}\sqrt[\tilde{k}]{n}}{1+\tilde{t}}F_{t,k}^2(0;v_0).$$
This completes the proof.
\end{proof}

By the definition of \eqref{nftk},
the indicatrix of $\varphi_{t,k}$  is defined by
\begin{equation}
I_{\varphi_{t,k}}(0)=\left\{v\in T_0^{1,0}P_m|\varphi_{t,k}(v)<1\right\}.
\end{equation}
Denote $B_m$ the unit ball of the canonical complex  Euclidean norm in $\mathbb{C}^m$. Then the following proposition is obtained immediately.
\begin{proposition}
Let $B_m, I_{\varphi_{t,k}}$ and $P_m$ be defined as above. Then
$$
B_m\subset I_{\varphi_{t,k}}\subset P_m.
$$
\end{proposition}

\begin{theorem}\label{th4} Let $f:P_m\rightarrow P_n$ be a holomorphic mapping from the unit polydisc $P_m$ into the unit polydisc $P_n$ satisfying $f(0)=0 $. Then
\begin{equation}
\tilde{\varphi}_{\tilde{t},\tilde{k}}^2\big(f(z)\big)\leq\frac{n+\tilde{t}\sqrt[\tilde{k}]{n}}{1+\tilde{t}}\varphi_{t,k}^2(z).
\label{4.1}
\end{equation}
\end{theorem}

\begin{proof}
Let $f(z)=\big(f_1(z),\cdots,f_n(z)\big)$ be a holomorphic mapping from $P_{m}$ to $P_n$. For any fixed $z_0=(z^1_0,\cdots,z^m_0)\in P_m\setminus\{0\}$, we set
$$\hat{z}_0=\frac{1}{\lambda_0}(z^1_0,\cdots,z^m_0),$$
where $\lambda_0=\displaystyle\max_{1\leq l\leq m}\big\{|z^l_0|\big\}\in(0,1)$.

Since $f_j(z)$ is holomorphic with respect to $z\in P_m$, $g_{j}(\zeta):=f_{j}(\zeta \hat{z}_0)$ is holomorphic  with respect to $\zeta$ from $\triangle$ into itself and $g_j(0)=f_j(0)=0$ for  $j=1,\cdots,n$. Thus $g_j$ has a Taylor expansion at $\zeta=0$, that is
\begin{equation}
	g_{j}(\zeta)=\sum_{i=N_j}^{\infty} b_{ji}\zeta^{i}\quad \mbox{with}\quad b_{ji}=\frac{g_j^{(i)}(0)}{i!}, \quad j=1, \cdots, n,\label{4.2}
\end{equation}
where $N_j\geq 1$ is the zero order of  $g_j$ at $\zeta=0$. 
Using the classical Schwarz lemma of one complex variable, it follows that
\begin{equation}
	|g_j(\zeta)|\leq |\zeta|^{N_j},\quad 
	 j=1,\cdots,n,\forall \zeta\in \triangle.\label{gz0}
\end{equation}
Note that
\begin{equation}
g_j(\lambda_0)=f_j(\lambda_0\hat{z}_0)=f_j(z_0),
\end{equation}
which together with \eqref{gz0} yields
\begin{equation}
	|f_j(z_0)|^2\leq \lambda_0^{2N_j}\leq\lambda_0^2,\quad \forall j=1, \cdots, n.\label{fz0}
\end{equation}
Therefore by \eqref{fz0} we have
\begin{eqnarray*}
	\tilde{\varphi}_{\tilde{t},\tilde{k}}^2\big(f(z_0)\big)
	 &=&\frac{1}{1+\tilde{t}}\left\{\sum_{j=1}^n|f_j(z_0)|^2+\tilde{t}\sqrt[\tilde{k}]{\sum_{j=1}^n|f_j(z_0)|^{2\tilde{k}}}\right\}\\
	&\leq&\frac{1}{1+\tilde{t}}\left\{\sum_{j=1}^n\lambda_0^2
	+\tilde{t}\sqrt[\tilde{k}]{\sum_{j=1}^n\lambda_0^{2\tilde{k}}}\right\}\\
	&=&\frac{n+\tilde{t}\sqrt[\tilde{k}]{n}}{1+\tilde{t}}\lambda_0^2.
\end{eqnarray*}
On the other hand, we have
\begin{eqnarray*}
\varphi_{t,k}^2(z_0)
&=&\frac{1}{1+t}\left\{\sum_{l=1}^{m}|z^{l}_0|^2+t\sqrt[k]{\sum_{l=1}^{m}|z^{l}_0|^{2k}} \right\}\\
&\geq&\frac{1}{1+t}\left\{\lambda_0^2+t\lambda_0^2\right\}\\
&=&\lambda_0^2.
\end{eqnarray*}
Therefore
$$
\tilde{\varphi}_{\tilde{t},\tilde{k}}^2\big(f(z_0)\big)\leq \frac{n+\tilde{t}\sqrt[\tilde{k}]{n}}{1+\tilde{t}}\varphi_{t,k}^{2}(z_0).
$$
It is clear that the above inequality  holds whenever $z_0=0$, hence holds for any $z\in P_m$.

To see that the constant $\frac{n+\tilde{t}\sqrt[\tilde{k}]{n}}{1+\tilde{t}}$ is optimal. Take $w=f_0(z)=\big(z^1,\cdots,z^1\big)$ and $z_0=(z_0^1, \underbrace{0, \cdots, 0}_{m-1}) \in P_m$ for any $|z_0^1|<1$. Then it is easy to check that
$$
\tilde{\varphi}_{\tilde{t},\tilde{k}}^2\big(f_{0}(z_0)\big)
=\frac{n+\tilde{t}\sqrt[\tilde{k}]{n}}{1+\tilde{t}}|z_0^1|^2=\frac{n+\tilde{t}\sqrt[\tilde{k}]{n}}{1+\tilde{t}}\varphi_{t,k}^2(z_0).
$$
This completes the proof.
\end{proof}

\begin{corollary}\label{c-5} Let $f:P_m\rightarrow P_n$ be a holomorphic mapping satisfying $f(\lambda z)=\lambda^Nf(z)$ for any $\lambda\in\triangle$ and an integer $N\geq 1$. Then
\begin{equation}
\tilde{\varphi}_{\tilde{t},\tilde{k}}^2\big(f(z)\big)\leq\frac{n+\tilde{t}\sqrt[\tilde{k}]{n}}{1+\tilde{t}}\varphi_{t,k}^{2N}(z).
\end{equation}
\end{corollary}
\begin{proof}By our assumption and the proof of Theorem \ref{th4}, it follows that $f(0)=0 $ and $g_j(\lambda\zeta)=\lambda^Ng_j(\zeta)$ for any $\lambda,\zeta\in\triangle$.
Thus $N_j\geq N$ for $j=1,\cdots,n$, and the assertion follows.
\end{proof}

\begin{proposition}\label{c-4} Let $f:P_m\rightarrow P_n$ be a holomorphic mapping satisfying $f(0)=0.$ Then for any fixed $l$~($l=1,\cdots,m$), the following assertions hold:
	
(1)The  equality
\begin{eqnarray}
\big(f^\ast\tilde{\varphi}^2_{\tilde{t},\tilde{k}}\big)(z_0)=\frac{n+\tilde{t}\sqrt[\tilde{k}]{n}}{1+\tilde{t}}\varphi_{t,k}^2(z_0),
\quad \forall 0\neq z_0=\big(\underbrace{0, \cdots, 0}_{l-1}, z_0^l, 0, \cdots, 0\big)\in P_m
\label{b-cb}
\end{eqnarray}	
holds
if and only if
\begin{eqnarray}
f_j\big(\underbrace{0, \cdots, 0}_{l-1}, z^l, 0, \cdots, 0\big)=e^{\sqrt{-1}\theta_j}z^l, \quad \theta_j\in \mathbb{R}, j=1, \cdots, n;	\label{b-cc}
\end{eqnarray}

(2) Suppose furthermore $f(\lambda z)=\lambda^Nf(z)$ for any $\lambda\in\triangle,z\in P_m$ and an integer $N\geq 1$, then
\begin{eqnarray}
\big(f^\ast\tilde{\varphi}^2_{\tilde{t},\tilde{k}}\big)(z_0)=\frac{n+\tilde{t}\sqrt[\tilde{k}]{n}}{1+\tilde{t}}\varphi_{t,k}^{2N}(z_0),\quad
\forall 0\neq z_0=\big(\underbrace{0, \cdots, 0}_{l-1}, z_0^l, 0, \cdots, 0\big)\in P_m
\label{b-cd}
\end{eqnarray}	
holds
if and only if
\begin{eqnarray}
f_j\big(\underbrace{0, \cdots, 0}_{l-1}, z^l, 0, \cdots, 0\big)=e^{\sqrt{-1}\theta_j}\big(z^l\big)^N, \quad \theta_j\in \mathbb{R}, j=1, \cdots, n. \label{b-ce}
\end{eqnarray}
\end{proposition}

\begin{proof}
(1) First, let's show the necessity. Using \eqref{fz0}, we have
\begin{eqnarray}
\big(f^\ast\tilde{\varphi}^2_{\tilde{t},\tilde{k}}\big)(z_0)
&=&\frac{1}{1+\tilde{t}}\left\{\sum_{j=1}^{n} \left|f_j(z_0)\right|^{2}+\tilde{t}\sqrt[\tilde{k}]{\sum_{j=1}^{n} \left|f_j(z_0)\right|^{2\tilde{k}}}\right\}\nonumber\\
&\leq& \frac{1}{1+\tilde{t}}\left\{\sum_{j=1}^{n}\lambda_0^{2N_j}+\tilde{t}\sqrt[\tilde{k}]{\sum_{j=1}^{n} \lambda_0^{2\tilde{k}N_j}}\right\}\nonumber\\
&\leq&\frac{n+\tilde{t}\sqrt[\tilde{k}]{n}}{1+\tilde{t}}\lambda_0^2\nonumber\\
&\leq& \frac{n+\tilde{t}\sqrt[\tilde{k}]{n}}{1+\tilde{t}}\varphi^2_{t,k}(z_0).\label{a-ra}
\end{eqnarray}
By \eqref{b-cb} and \eqref{a-ra}, we have
$$\left|f_j(z_0)\right|=|z_0^l|^{N_j}=|z_0^l|,\label{a-rc}$$
thus we have $ N_j=1.$ By the classical Schwarz lemma of one complex variable, we have
$$f_j\big(\underbrace{0, \cdots, 0}_{l-1}, z^l, 0, \cdots, 0\big)=e^{\sqrt{-1}\theta_j}z^l, \quad \theta_j\in \mathbb{R}, j=1, \cdots, n.$$
	
Next, we show the sufficiency. If $f_j\big(\underbrace{0, \cdots, 0}_{l-1}, z^l, 0, \cdots, 0\big)=e^{\sqrt{-1}\theta_j}z^l, \theta_j\in \mathbb{R}, j=1, \cdots, n$ hold. Taking $z_0=\big(\underbrace{0, \cdots, 0}_{l-1}, z_0^l, 0, \cdots, 0\big)\neq0,$ then we have $$f(z_0)=\big(e^{\sqrt{-1}\theta_1}z_0^l, \cdots, e^{\sqrt{-1}\theta_n}z_0^l\big)$$
and
$$
\big(f^\ast\tilde{\varphi}^2_{\tilde{t},\tilde{k}}\big)(z_0)
=\frac{|z_0^l|^{2}}{1+\tilde{t}}\left\{\sum_{j=1}^{n}\big|e^{\sqrt{-1}\theta_j}\big|^{2}+\tilde{t}\sqrt[\tilde{k}]{\sum_{j=1}^{n} \big|e^{\sqrt{-1}\theta_j}\big|^{2\tilde{k}}}\right\}
=\frac{n+\tilde{t}\sqrt[\tilde{k}]{n}}{1+\tilde{t}}|z_0^l|^{2}, \quad \varphi_{t,k}^2(z_0)=|z_0^l|^{2}.
$$
Consequently
$$\big(f^\ast\tilde{\varphi}^2_{\tilde{t},\tilde{k}}\big)(z_0)=\frac{n+\tilde{t}\sqrt[\tilde{k}]{n}}{1+\tilde{t}}\varphi_{t,k}^2(z_0).$$
This completes the proof of assertion (1).
	
(2) First, we show the necessity. By our assumption and Corollary \ref{c-5}, we have $f(0)=0$ and $g_j(\lambda\zeta)=\lambda^Ng_j(\zeta)$ for any $\lambda,\zeta\in\triangle$.
Thus $N_j\geq N$ for $j=1,\cdots,n$, and \eqref{a-ra} could be replaced by
\begin{eqnarray}
\big(f^\ast\tilde{\varphi}^2_{\tilde{t},\tilde{k}}\big)(z_0)
&=&\frac{1}{1+\tilde{t}}\left\{\sum_{j=1}^{n} \left|f_j(z_0)\right|^{2}+\tilde{t} \sqrt[\tilde{k}]{\sum_{j=1}^{n} \left|f_j(z_0)\right|^{2\tilde{k}}}\right\}\nonumber\\
&\leq& \frac{1}{1+\tilde{t}}\left\{\sum_{j=1}^{n}\max_{1\leq i\leq m}\big\{|z_0^i|^{2N}\big\}+\tilde{t}\sqrt[\tilde{k}]{\sum_{j=1}^{n} \max_{1\leq i\leq m}\big\{|z_0^i|^{2\tilde{k}N}\big\}}\right\}\nonumber\\
&=& \frac{1}{1+\tilde{t}}\left\{\sum_{j=1}^{n}|z_0^l|^{2N}+\tilde{t}\sqrt[\tilde{k}]{\sum_{j=1}^{n} |z_0^l|^{2\tilde{k}N}}\right\}\nonumber\\
&=&\frac{n+\tilde{t}\sqrt[\tilde{k}]{n}}{1+\tilde{t}}|z_0^l|^{2N}.\label{a-rb}
\end{eqnarray}
By \eqref{b-cd} and \eqref{a-rb}, we have
$$\left|f_j(z_0)\right|=|z_0^l|^{N},\label{a-rd}$$
Denote $h_j\big(z^l\big):=f_j\big(\underbrace{0, \cdots, 0}_{l-1}, z^l, 0, \cdots, 0\big),$ then $\left|h_j(z^l_0)\right|=|z_0^l|^{N}.$ Since $h_j\big(\lambda z^{l}\big)=\lambda^{N}h_j\big(z^l\big)$, by taking $i$-th order derivatives of $h_j(\lambda z^l)$ with respect to $\lambda$, $i=0,\cdots,N-1,$ and setting $\lambda=0,$ we have $h_j(0)=h'_j(0)=\cdots=h^{(N-1)}_j(0)=0$. By the classical Schwarz lemma of one complex variable, we have
$$f_j\big(\underbrace{0, \cdots, 0}_{l-1}, z^l, 0, \cdots, 0\big)=h_j(z^l)=e^{\sqrt{-1}\theta_j}\big(z^l\big)^N, \quad \theta_j\in \mathbb{R}, j=1, \cdots, n.$$
	
Next, we show the sufficiency. If $f_j\big(\underbrace{0, \cdots, 0}_{l-1}, z^l, 0, \cdots, 0\big)=e^{\sqrt{-1}\theta_j}(z^l)^N, \theta_j\in \mathbb{R}, j=1, \cdots, n$ hold. Taking $z_0=\big(\underbrace{0, \cdots, 0}_{l-1}, z_0^l, 0, \cdots, 0\big)\neq0,$ then we have $$f(z_0)=\Big(e^{\sqrt{-1}\theta_1}(z_0^l)^N, \cdots, e^{\sqrt{-1}\theta_n}(z_0^l)^N\Big)$$
and
$$
\big(f^\ast\tilde{\varphi}^2_{\tilde{t},\tilde{k}}\big)(z_0)
=\frac{|z_0^l|^{2N}}{1+\tilde{t}}\left\{\sum_{j=1}^{n}\big|e^{\sqrt{-1}\theta_j}\big|^{2}+\tilde{t}\sqrt[\tilde{k}]{\sum_{j=1}^{n} \big|e^{\sqrt{-1}\theta_j}\big|^{2\tilde{k}}}\right\}
=\frac{n+\tilde{t}\sqrt[\tilde{k}]{n}}{1+\tilde{t}}|z_0^l|^{2N}, \quad \varphi_{t,k}^2(z_0)=|z_0^l|^{2}.
$$
Consequently
$$\big(f^\ast\tilde{\varphi}^2_{\tilde{t},\tilde{k}}\big)(z_0)=\frac{n+\tilde{t}\sqrt[\tilde{k}]{n}}{1+\tilde{t}}\varphi_{t,k}^{2N}(z_0).$$
This completes the proof.
\end{proof}

\section{Distortion theorem on polydiscs endowed with K\"ahler-Berwald metrics } \label{se4}

The distortion theorem is an important research object in  the classical geometric function theories of one complex variable,
as a corollary of the De Brange theorem,
 for any normalized biholomorphic function $f$ on the unit disk $\Delta$   with $f(0)=0$ and $f^{\prime}(0)=1$, the following assertions hold:

(1)$\frac{1-|z|}{(1+|z|)^{3}} \leqslant\left|f^{\prime}(z)\right| \leqslant \frac{1+|z|}{(1-|z|)^{3}}, \quad \forall z \in \Delta $;

(2)the above inequalities hold at some $z\neq0$ if and only if $f(z)=\frac{z}{(1-e^{i\theta}z)^2}, (\theta\in \mathbb{R}).$

In the case of several complex variables, however, there are many counter examples to show that the distortion theorem does not hold.
For example \cite{Cartan},
$$f(z)=\left(z^{1}, \frac{z^{2}}{\left(1-z^{1}\right)^{n}}\right), \quad z=(z^1,z^2)\in B_2, n \in \mathbb{N}^{+} .$$
It is easy to check that $f$ is biholomorphic, $f(0)=0$ and $f_{\ast}(0)=I_{2}$, i.e., $f$ is a normalized biholomorphic mapping, where $f_{\ast}(0)$ is the Jacobian matrix of $f$ at the origin and $I_{2}$ is the identity matrix. But $\operatorname{det} f_{\ast}(z)$ and $f_{\ast}(z) \overline{f_{\ast}(z)^{\prime}}$ for the holomorphic mappings $f$ have neither finite upper bound nor non-zero lower bound, where $f_{\ast}(z)$ is the Jacobian matrix of $f$ at point $z$. In order to extend the above classical distortion theorem to high dimensions, H. Cartan \cite{Cartan} suggested the study of convex mappings, starlike mappings and some other subclasses of biholomorphic mappings in several complex variables.

By the distortion theorem of K.L\"{o}wner,
if $f(z)=z+\sum\limits_{i=2}^{\infty}a_iz^i$ is a biholomorphic convex function on the unit disk $\Delta$ with $f(0)=0$ and $f^{\prime}(0)=1$, then $|a_i|\leq1,\,i=2, 3, \cdots,$ and $|a_i|=1$ for some $i\geq2$ if and only if $f(z)=\frac{z}{1-e^{i\theta}z}~(\theta\in \mathbb{R}).$ Morever, the following inequalities hold
	
	(1) $ \frac{1}{(1+|z|)^2}\leq|f'(z)|\leq\frac{1}{(1-|z|)^2}$;
	
	(2) the above inequalities hold at some $z\neq0$ if and only if $f(z)=\frac{z}{1-e^{i\theta}z}\; (\theta\in \mathbb{R}).$

In \cite{Gong-W-Y}, S. Gong, S. K. Wang and Q. H. Yu gave the following precise estimations of $f_{\ast}(z) \overline{f_{\ast}(z)}{ }^{\prime}$ for the family of normalized biholomorphic convex mappings on the unit ball $B_{m}$ of $\mathbb{C}^m$.

\begin{theorem}\cite{Gong-W-Y}
Let $f: B_{m} \rightarrow \mathbb{C}^{m}$ be a normalized biholomorphic convex mapping on the unit ball $B_{m}\subset \mathbb{C}^{m}$. Then
\begin{equation}
\left(\frac{1-\|z\|}{1+\|z\|}\right)^2 Q \leq f_{\ast}(z) \overline{f_{\ast}(z)}' \leq\left(\frac{1+\|z\|}{1-\|z\|}\right)^2 Q\label{d-t}
\end{equation}
holds for every $z=\left(z^{1}, \cdots, z^{m}\right) \in B_{m}$, where
\begin{equation}
Q=\left(q_{ij}\right)_{1 \leq i, j \leq m}=\left(\frac{\left(1-\|z\|^{2}\right) \delta_{i j}+\bar{z}_{i} z_{j}}{\left(1-\|z\|^{2}\right)^{2}}\right)_{1 \leq i, j \leq m}\label{Ball}
\end{equation}
is the Hermitian tensor matrix of the Bergman matric of $B_{m}$, and $\|.\|^{2}$ is the canonical complex Euclidean norm in $\mathbb{C}^{m}$. The estimations are precise.
\end{theorem}
\begin{remark}
Two Hermitian matrices $A\leq B$ in \eqref{d-t} means the difference $A-B$ is a seminegative definite Hermitian matrix. It is clear that \eqref{d-t} is equivalent to the following inequalities
$$
\left(\frac{1-\|z\|}{1+\|z\|}\right)^2Q(z;v)\leq \|f_\ast(z)v\|^2\leq \left(\frac{1+\|z\|}{1-\|z\|}\right)^2Q(z;v),\quad \forall z \in B_m, v \in T_z^{1,0}B_m\cong\mathbb{C}^{m},
$$
where $Q(z;v)$ is the Bergman metric on $B_m$, with $Q$ given by \eqref{Ball}.
\end{remark}
\begin{definition}\cite{Shi}\label{def1}
Let $\Omega$ be a bounded convex circular domain in $\mathbb{C}^m$, define
$$
p(z)=\inf\left\{c>0:\frac{z}{c}\in\Omega\right\},\quad z\in\mathbb{C}^m,
$$
then (i) $p(z)$ is a  complex norm on $\mathbb{C}^m$, called the Minkowski functional of $\Omega$;
(ii) under the complex norm $p(\cdot)$, $\mathbb{C}^m$ is a complex Banach space;
(iii) $\Omega=\{z\in\mathbb{C}^m: p(z)<1\}$,  $\partial\Omega=\{z\in\mathbb{C}^m: p(z)=1\}$.
\end{definition}
It follows from Definition \ref{def1} that the Minkowski functional of the ploydisc $ P_m $ is given by $ p(z)=\max\limits_{1\leq l\leq m}\{|z^l|\} $ for any $z\in \mathbb{C}^m.$

In \cite{Gong-Liu}, S. Gong and T. S. Liu generalized the distortion theorem of biholomorphic convex mappings to bounded circular domains in terms of Carth\'{e}odory metric and Kobayashi metric in $\mathbb{C}^m.$
\begin{theorem}\label{th6}\cite{Gong-Liu}
Let $\Omega \subset \mathbb{C}^{m}$ be a bounded convex circular domain with $0 \in \Omega$, and $p(z)(z \in \Omega)$ be its Minkowski functional. Let $f(z): \Omega \rightarrow \mathbb{C}^{m}$ be a normalized biholomorphic convex mapping on $\Omega$. Then for every $z \in \Omega$ and vector $v\in \mathbb{C}^{m}$, the inequalities
$$
\frac{1-p(z)}{1+p(z)} F^{\Omega}(z; v) \leq p\left(f_{\ast}(z) v\right) \leq \frac{1+p(z)}{1-p(z)} F^{\Omega}(z; v)
$$
hold, where
$$
F^{\Omega}(z; v)=F_{C}^{\Omega}(z; v)=F_{K}^{\Omega}(z; v),
$$
$F_{C}^{\Omega}(z; v)$ and $F_{K}^{\Omega}(z; v)$ are the infinitesimal form of the Carath\'{e}odory metric and the infinitesimal form of the Kobayashi-Royden metric of $\Omega$ respectively.
\end{theorem}

\begin{remark}\label{bi-con}	
In \cite{Su}, T. J. Suffridge obtained the following decomposition theorem: Suppose $f(z)=\left(f_{1}(z), \cdots, f_{m}(z)\right), z\in \mathbb{C}^{m}$ is holomorphic on the polydisc $P_{m}, f(0)=0.$ Then $f(z)$ is a biholomorphic convex mapping if and only if there exists a nonsingular constant matrix $T$ such that $f(z)=\left(\phi_{1}\left(z^{1}\right), \cdots, \phi_{m}\left(z^{m}\right)\right) T$ holds, where $\phi_{l}\left(z^{l}\right)$ is a biholomorphic convex function of one complex variable $z^{l}$ on the unit disc $\Delta_l,l=1,\cdots,m.$
	
Note that if $f$ is a normalized biholomorphic convex mapping, then $f(0)=0$ and $f_{\ast}(0)=I_m$.  Hence we have  $\phi_{l}\left(0\right)=0, l=1, \cdots, m,$ and
$$ T= \left(\begin{array}{ccc}
			\frac{1}{\phi'_1(0)}&\cdots&0\\
		    \vdots&\ddots&\vdots\\
			0&\cdots&\frac{1}{\phi'_m(0)}\end{array}
\right).$$
Then we have $f(z)=\left(f_{1}\left(z\right), \cdots, f_{m}\left(z\right)\right)=\left(\frac{\phi_{1}\left(z^{1}\right)}{\phi'_1(0)}, \cdots, \frac{\phi_{m}\left(z^{m}\right)}{\phi'_m(0)}\right).$
Thus $f(z)$ is a normalized biholomorphic convex mapping if and only if $f_l(z)=f_{l}\left(z^{l}\right)$ is a normalized biholomorphic convex function of one complex variable $z^{l}$ on the unit disc $\Delta_l, l=1, \cdots, m.$		
\end{remark}

In terms of the $\mbox{Aut}(P_m)$-invariant K\"ahler-Berwald metric $F_{t,k}$ given by \eqref{Ftk}, which is not necessary Hermitian quadratic, we are able to obtain the following distortion theorem for the normalized biholomorphic convex mappings on $P_m$.
\begin{theorem}\label{thdis}
Let $P_m$ be the unit polydisc in $\mathbb{C}^m$ endowed with the $\mbox{Aut}(P_m)$-invariant K\"ahler-Berwald metric $F_{t,k}$ given by \eqref{Ftk}, and $f(z)=(f_1(z),\cdots,f_n(z)): P_m \rightarrow \mathbb{C}^{m}$ be a normalized biholomorphic convex mapping.	
Then for every $z \in P_m$ and  $v \in T_z^{1,0}P_m\cong\mathbb{C}^{m}$, the inequalities
\begin{equation}
\left[\frac{1-p(z)}{1+p(z)}\right]^2F_{t,k}^{2}(z;v)
\leq F_{t,k}^2\left(0; f_{\ast}(z) v\right)
\leq \left[\frac{1+p(z)}{1-p(z)}\right]^2F_{t,k}^{2}(z; v)
\label{disto}
\end{equation}
hold, here $p(z)$ is the Minkowski functional of $P_m$.
Moreover, one of the equalities holds at some point $z\neq0$ in \eqref{disto} if and only if $f_l(z)=\frac{z^l}{1-e^{i\theta_l}z^l}$ for some $\theta_l\in \mathbb{R}, l=1,\cdots,m.$
\end{theorem}
\begin{proof}
Since $f(z)=(f_1(z),\cdots,f_n(z)): P_m \rightarrow \mathbb{C}^{m}$ is a normalized biholomorphic convex mapping, it follows from T. J. Suffridge's decomposition theorem \cite{Su} that $f_l(z)=f_{l}\left(z^{l}\right)$ is a normalized biholomorphic convex function of one complex variable $z^{l}$ on the unit disc $\Delta_l$ for $l=1, \cdots, m$. Thus we have	
\begin{equation}
	F_{t,k}^2\left(0; f_{\ast}(z) v\right)
	=\frac{1}{1+t}\left\{\sum_{l=1}^{m}|f_{l}'(z^l)v^l|^2+t\sqrt[k]{\sum_{l=1}^{m}|f_{l}'(z^l)v^l|^{2k}}
	\right\},\quad  \forall z \in P_m, v \in T_z^{1,0}P_m\cong\mathbb{C}^m.	
	\label{disto1}
\end{equation}
By the result of K. Lowner, we have
\begin{eqnarray}
\frac{1}{(1+|z^l|)^2}\leq |f_l'(z^l)|\leq \frac{1}{(1-|z^l|)^2},\quad l=1, \cdots, m.
\label{f_l'(z^l)}
\end{eqnarray}
Now it follows from  \eqref{disto1} and \eqref{f_l'(z^l)} that
\begin{eqnarray}
F_{t,k}^2\left(0; f_{\ast}(z)v\right)
&\leq& \frac{1}{1+t}\left\{\sum_{l=1}^{m} \frac{|v^l|^2}{(1-|z^l|)^4}+t\sqrt[k]{\sum_{l=1}^{m} \frac{|v^l|^{2k}}{(1-|z^l|)^{4k}}}\right\}\label{disto2}
\end{eqnarray}
and
\begin{eqnarray}
F_{t,k}^2\left(0; f_{\ast}(z) v\right)
&\geq& \frac{1}{1+t}\left\{\sum_{l=1}^{m} \frac{|v^l|^2}{(1+|z^l|)^4}+t\sqrt[k]{\sum_{l=1}^{m} \frac{|v^l|^{2k}}{(1+|z^l|)^{4k}}}\right\}.
\label{disto2'}
\end{eqnarray}
Since
\begin{eqnarray} \frac{1}{(1-|z^l|)^4}=\frac{1}{(1-|z^l|^2)^2}\left[\frac{1+|z^l|}{1-|z^l|}\right]^2 \leq\frac{1}{(1-|z^l|^2)^2}\left[\frac{1+p(z)}{1-p(z)}\right]^2 \label{3'},
\end{eqnarray}
which together with \eqref{disto2} implies
\begin{eqnarray*}
F_{t,k}^2\left(0; f_{\ast}(z)v\right)
&\leq& \frac{1}{1+t}\left[\frac{1+p(z)}{1-p(z)}\right]^2\left\{\sum_{l=1}^{m} \frac{|v^l|^2}{(1-|z^l|^2)^2}+t\sqrt[k]{\sum_{l=1}^{m} \frac{|v^l|^{2k}}{(1-|z^l|^2)^{2k}}}\right\}\nonumber\\
&=&\left[\frac{1+p(z)}{1-p(z)}\right]^2F_{t,k}^2\left(z;v\right).
\end{eqnarray*}
Similarly, since
\begin{eqnarray} \frac{1}{(1+|z^l|)^4}=\frac{1}{(1-|z^l|^2)^2}\left[\frac{1-|z^l|}{1+|z^l|}\right]^2\geq\frac{1}{(1-|z^l|^2)^2}\left[\frac{1-p(z)}{1+p(z)}\right]^2, \label{4'}
\end{eqnarray}
which together with  \eqref{disto2'} implies
\begin{eqnarray*}
F_{t,k}^2\left(0; f_{\ast}(z) v\right)
&\geq& \frac{1}{1+t}\left[\frac{1-p(z)}{1+p(z)}\right]^2\left\{\sum_{l=1}^{m} \frac{|v^l|^2}{(1-|z^l|^2)^2}+t\sqrt[k]{\sum_{l=1}^{m} \frac{|v^l|^{2k}}{(1-|z^l|^2)^{2k}}}\right\}\nonumber\\
&=& \left[\frac{1-p(z)}{1+p(z)}\right]^2F_{t,k}^2\left(z;v\right).
\end{eqnarray*}
Thus \eqref{disto} is proved. Moreover, if one of the equalities holds at some point $z\neq0$ in \eqref{disto}, then
\begin{eqnarray*}
\frac{1}{(1+|z^l|)^2}= |f_l'(z^l)|\label{|f_l'(z^l)|1}\quad \mbox{or}\quad
\frac{1}{(1-|z^l|)^2}=|f_l'(z^l)|, \quad l=1,\cdots, m.\label{|f_l'(z^l)|}
\end{eqnarray*}
By the result of K. Lowner, we have
$f_l(z^l)=\frac{z^l}{1-e^{i\theta_l}z^l}$, where $\theta_l\in \mathbb{R}$ for  $l=1,\cdots,m$. On the other hand, if $f_l(z^l)=\frac{z^l}{1-e^{i\theta_l}z^l}$ with $\theta_l\in \mathbb{R}$ for $l=1,\cdots,m$, and if we take $z_0=(be^{-i\theta_1},\cdots,be^{-i\theta_m})$ with $b\in \mathbb{R^+}$, then $|f_l'(z_0^l)|^2= \frac{1}{(1-b)^4}$ and \eqref{3'}  reduces to
 \begin{eqnarray}
 \frac{1}{(1-b)^4}= \frac{1}{(1-b^2)^2}\left[\frac{1+b}{1-b}\right]^2= \frac{1}{(1-|z_0^l|^2)^2}\left[\frac{1+p(z_0)}{1-p(z_0)}\right]^2.\label{|f_l'(z_0^l)|}
 \end{eqnarray}
Hence
\begin{equation*}
	 F_{t,k}^2\left(0, f_{\ast}(z_0)v\right)
	=\left[\frac{1+p(z_0)}{1-p(z_0)}\right]^2F_{t,k}^{2}(z_0;v).
\end{equation*}
Similarly, for $\tilde{ z}_0=(-ae^{-i\theta_1},\cdots,-ae^{-i\theta_m})$ with $a\in \mathbb{R^+}$,  we have
\begin{equation*}
 F_{t,k}^2\left(0, f_{\ast}(\tilde{ z}_0) v\right)=\left[\frac{1-p(\tilde{ z}_0)}{1+p(\tilde{ z}_0)}\right]^2F_{t,k}^{2}(\tilde{ z}_0;v).
\end{equation*}
This completes the proof.
\end{proof}

\begin{corollary}\label{dis-corollary}
Let $P_m$ be the unit polydisk in $\mathbb{C}^m$ endowed with the $\mbox{Aut}(P_m)$-invariant K\"ahler-Berwald metric $F_{t,k}$ which is given by \eqref{Ftk}, and $f(z): P_m \rightarrow \mathbb{C}^{m}$ with $f_l(z)=f_{l}\left(z^{l}\right)$ being a normalized biholomorphic convex function of the complex variable $z^{l}$ on the unit disc $\Delta_l(l=1,\cdots,m).$	
Then
\begin{equation}
		\frac{p(z)^2}{(1+p(z))^4}
		\leq F_{t,k}^2\left(0; f_{\ast}(z)z\right)
		\leq \dfrac{m+t\sqrt[k]{m}}{1+t}\frac{p(z)^2}{(1-p(z))^4}\label{disto-c}
\end{equation}
for every $z \in P_m.$ Moreover, the right hand side equality holds at some $z_0$ if and only if $z_0=(be^{i\theta_1}, \cdots, be^{i\theta_m})$ with $b\in \mathbb{R^+}$ and $f_l(z^l)=\frac{z^l}{1-e^{i\theta_l}z^l}$ for some $\theta_l\in \mathbb{R}, l=1,\cdots,m.$
\end{corollary}

\begin{proof}
According to the definition of the $\mbox{Aut}(P_m)$-invariant K\"ahler-Berwald metric $F_{t,k}$ given by \eqref{Ftk}, we have
\begin{equation}
\dfrac{p(z)^2}{(1-p(z)^2)^2}\leq F_{t,k}^2\left(z; z\right)\leq \dfrac{m+t\sqrt[k]{m}}{1+t}\dfrac{p(z)^2}{(1-p(z)^2)^2}, \quad \forall z \in P_m.\label{pftk}
\end{equation}
Thus by \eqref{pftk} and  replacing $v$ with $z$ in Theorem \ref{thdis}, we have
\begin{equation*}
\frac{p(z)^2}{(1+p(z))^4}
\leq F_{t,k}^2\left(0; f_{\ast}(z)z\right)
\leq \dfrac{m+t\sqrt[k]{m}}{1+t}\frac{p(z)^2}{(1-p(z))^4},
\end{equation*}
which completes the proof of \eqref{disto-c}. Moreover, the right hand side equality in \eqref{disto-c} holds at some point $z_0$ is equivalent to
\begin{eqnarray}
F_{t,k}^2\left(0; f_{\ast}(z_0) z_0\right)
&=&\left[\frac{1+p(z_0)}{1-p(z_0)}\right]^2F_{t,k}^{2}(z_0;z_0)\label{dis-eq3}
\end{eqnarray}
and
\begin{eqnarray}
F_{t,k}^2\left(z_0; z_0\right)&=&\dfrac{m+t\sqrt[k]{m}}{1+t}\dfrac{p(z_0)^2}{(1-p(z_0)^2)^2}.\label{dis-eq4}
\end{eqnarray}
Thus by Theorem \ref{thdis}, \eqref{dis-eq3} holds at $z_0$
if and only if $f_l(z^l)=\frac{z^l}{1-e^{i\theta_l}z^l}$ for some $\theta_l\in \mathbb{R}, l=1,\cdots,m.$ By \eqref{Ftk}, it follows that \eqref{dis-eq4} holds at $z_0$ if and only if $\frac{|z_0^i|}{1-|z_0^i|^2}=\frac{|z_0^j|}{1-|z_0^j|^2}$ for $i,j=1,\cdots,m$, if and only if $z_0=(be^{i\theta_1},\cdots,be^{i\theta_m})$ with $b\in \mathbb{R^+}$.
This completes the proof.
\end{proof}
\section{Complex Finsler-Einstein condition}

In this section, we show that the complex Finsler metric defined in \eqref{Ftk} is a complex Finsler-Einstein metric in the sense of T. Aikou \cite{Ai2}.

Let $M$ be a complex manifold of complex dimension $n$ and $g=g_{\alpha\bar{\beta}}dz^\alpha d\bar{z}^\beta$ be a Hermitian metric on $M$. Let $E$ be a holomorphic vector bundle of rank $r$ over $M$ whose fiber metric is a strongly pseudoconvex complex Finsler metric $F(z;\xi)$. Denote $R_{j\alpha\bar{\beta}}^i$ the curvature components of the partial connection (or equivalently the horizontal Chern-Finsler connection) associated to $F$. Define the partial mean curvature $K_j^i$ as follows \cite{Ai2}:
\begin{equation}
K_j^i:=g^{\bar{\beta}\alpha}R_{j\alpha\bar{\beta}}^i.\label{mc}
\end{equation}
In general, $R_{j\alpha\bar{\beta}}^i$ depend on both local coordinates $z$ of the base manifold $M$ and the fiber coordinates $\xi$, and it is easy to check that $K_j^i$ are actually well-defined on the projective bundle $PE$ of $E$.
In \cite{SZ2}, it was proved that if
$$K_j^i=\varphi(z,\xi)\delta_j^i$$
for some function $\varphi(z,\xi)$ on the projective bundle $PE$, then $\varphi(z,\xi)=\varphi(z)$.
Thus the following definition is natural.
\begin{definition}
A complex Finsler bundle $(E,F)$ is said to be a weakly Finsler-Einstein vector bundle if the partial mean curvature $K=(K_j^i)$ satisfies
$$
K_j^i=\varphi(z)\delta_j^i,
$$
for a function $\varphi$ on $M$. If the factor $\varphi$ is constant, $(E,F)$ is said to be a complex Finsler-Einstein vector bundle.
\end{definition}
It is easy to check that the above definition is equivalent to the Definition 3.1 in T. Aikou \cite{Ai2}. There are also other notions of complex Finsler-Einstein vector bundle, we refer to \cite{K-2} and some recent progress in \cite{FLW1,FLW2} along this line.

In the case that $E=T^{1,0}M$ is the holomorphic tangent bundle of $M$, it seems that Aikou's definition has a little awkward. Nevertheless, by Theorem 4.1 in \cite{Ai2}, a complex Berwald manifold $(M,F)$ is necessary a Hermitian manifold, more precisely, there exists a Hermitian metric $h_M$ on $M$ such that the partial connection associated to $F(z;v)$ is equal to the pull-back of the Hermitian connection of $h_M$. Thus one can always take $h_M$ instead of $g$  as definition in \eqref{mc},  which obviously makes sense for complex Berwald manifolds. Under this convention, a complex Berwald manifold $(M,F)$ is called a complex Finsler-Einstein manifold if it satisfies \eqref{mc}; a K\"ahler-Berwald manifold $(M,F)$ is called a K\"ahler Finsler-Einstein manifold if it satisfies \eqref{mc}. And if $(M,F)$ is a complex Finsler-Einstein manifold, then $(M,h_M)$ is necessary a Hermitian-Einstein manifold; if $(M,F)$ is a K\"ahler Finsler-Einstein manifold, then $(M,h_M)$ is necessary a K\"ahler-Einstein manifold. The following theorem shows that $(P_m,F_{t,k})$ is a K\"ahler Finsler-Einstein manifold for any fixed $t\in[0,+\infty)$ and integer $k\geq 2$, where $F_{t,k}$ is not necessary a Hermitian quadratic metric.

\begin{theorem}\label{Einstein}
Let $P_{m}$ be the  unit polydisc in $\mathbb{C}^{m}$ endowed with the $\mbox{Aut}(P_m)$-invariant K\"ahler-Berwald metric $F_{t, k}$ given by \eqref{Ftk}.
 Then for any fixed $t\in[0,+\infty)$ and integer $k\geq 2$, $(P_m,F_{t,k})$ is a complex Finsler-Einstein manifold with factor $\varphi=-2$.
\end{theorem}

\begin{proof}
Note that the Bergman metric
$$g=\sum_{l=1}^mg_l, \quad g_l=\frac{dz^ld\bar{z}^l}{(1-|z^l|^2)^2},$$
is a K\"ahler-Einstein metric on $P_m$, which is a Hermitian quadratic metric.

In \cite{Zh2}, the horizontal Chern-Finsler connection coefficients $\varGamma^{s}_{l;i}$ associated to
the complex Finsler metric $F_{t,k}$  are given by
\begin{equation*}
\varGamma^{s}_{l;i}
=\left\{
 \begin{array}{ll}
\hat{\varGamma}^{l}_{l;l}, &\mbox{if}\;l=i=s; \\
0, & \hbox{otherwise},
 \end{array}
 \right.
\end{equation*}
where $\hat{\varGamma}^{l}_{l;l}$ is the Hermitian connection coefficient of $g_{l}$ for $l=1,\cdots,m$, respectively. Thus  the curvature components $R^{{s}}_{{l}{i}\bar{r}}$ of the partial connection associated to $F_{t,k}$ are given by
\begin{equation*}
R^{{s}}_{{l}{i}\bar{r}}
=\left\{
 \begin{array}{ll}
\hat{R}^{{l}}_{{l}{l}\bar{l}}, &\mbox{if}\;i=r=s=l; \\
0, & \hbox{otherwise},
 \end{array}
 \right.
\end{equation*}
where $\hat{R}^{{l}}_{{l}{l}\bar{l}}
=-\frac{\partial}{\partial \bar{z}^{l}}\hat{\varGamma}^{l}_{{l};l}
=\frac{-2}{(1-|z^{l}|^{2})^{2}}
$
are the curvature components of the Hermitian connection associated to  $g_{l}$. It follows that  the mean curvature $K=(K^{s}_{l})$ of $F_{t,k}$ with respect to $g$ are given by
\begin{eqnarray*}
&&K^{s}_{l}
=g^{\bar{r}{i}}R^{s}_{l{i}\bar{r}}=-2\delta^{s}_l,
\end{eqnarray*}
which implies that  $(P_m,F_{t,k})$ is a complex Finsler-Einstein  manifold with factor $\varphi=-2$.

This completes the proof of Theorem \ref{Einstein}.
\end{proof}

\vskip0.4cm
{\bf Acknowledgement:}\ {\small This work is supported by the National Natural Science Foundation of China (Grant No. 12071386, 11671330)}.

\end{document}